\documentclass[11pt,letterpaper]{amsart}

\usepackage[applemac]{inputenc}
\usepackage{amsfonts} 
\usepackage{amsmath} 
\usepackage{amsthm} 
\usepackage{xcolor} 
\usepackage{hyperref}
\usepackage{todonotes}
\usepackage{amssymb}
\usepackage[all]{xypic}

\makeatletter
\@namedef{subjclassname@2020}{\textup{2020} Mathematics Subject Classification}
\makeatother

\title[Weakly K\"ahler hyperbolic manifolds and the GGL conjecture]{Weakly K\"ahler hyperbolic manifolds and the Green--Griffiths--Lang conjecture}
\author[F. Bei\and S. Diverio\and P. Eyssidieux\and S. Trapani]{Francesco Bei \and Simone Diverio \and Philippe Eyssidieux\and Stefano Trapani}
\address{Francesco Bei and Simone Diverio \\ Dipartimento di Matematica \lq\lq Guido Castelnuovo\rq\rq{} \\ SAPIENZA Universit\`a di Roma \\ Piazzale Aldo Moro 5 \\ I-00185 Roma.}
\email{bei@mat.uniroma1.it \\ diverio@mat.uniroma1.it } 
\address{Philippe Eyssidieux \\ Université Grenoble Alpes \\ CNRS \\ IF \\ 38000 Grenoble\\ France} 
\email{philippe.eyssidieux@univ-grenoble-alpes.fr} 
\address{Stefano Trapani \\  Università di Roma \lq\lq Tor Vergata\rq\rq{} \\ Via della Ricerca Scientifica 1 \\ I-00133 Roma.} 
\email{trapani@mat.uniroma2.it} 
\keywords{Kähler hyperbolic manifold,  weakly Kähler hyperbolic manifold, spectral gap, Atiyah $L^2$-theory, Lang's conjecture, Kobayashi hyperbolicity, Ahlfors' current, non-Kähler and null locus, pluripotential theory, variety of general type}
\subjclass[2020]{Primary: 32Q15; Secondary: 32J25, 32J27, 32Q45, 53C55.}
\thanks{The first-named author is partially supported by the \lq\lq Gruppo Nazionale per le Strutture Algebriche, Geometriche e le loro Applicazioni\rq\rq{} of the Istituto Nazionale di Alta Matematica \lq\lq Francesco Severi\rq\rq{}.}
\thanks{The second-named author is partially supported by the ANR Programme: D\'efi de tous les savoirs (DS10) 2015, \lq\lq GRACK\rq\rq, Project ID: ANR-15-CE40-0003ANR, and by the ANR Programme: D\'efi de tous les savoirs (DS10) 2016, \lq\lq FOLIAGE\rq\rq, Project ID: ANR-16-CE40-0008. He is also partially supported by the \lq\lq Gruppo Nazionale per le Strutture Algebriche, Geometriche e le loro Applicazioni\rq\rq{} of the Istituto Nazionale di Alta Matematica \lq\lq Francesco Severi\rq\rq{} as well as the \lq\lq SEED PNR\rq\rq{} project of SAPIENZA Universit\`a di Roma}
\thanks{The third-named author  was partially supported by the ANR project Hodgefun ANR-16-CE40-0011-01.}
\thanks{The fourth-named author is supported by the \lq\lq Excellence project\rq\rq{} of the Italian ministry of University and Research}
\date{\today}
\dedicatory{In loving memory of Jean-Pierre Demailly}

\theoremstyle{plain}
\newtheorem{thm}{Theorem}[section]
\newtheorem{cor}[thm]{Corollary}
\newtheorem{bigthm}{Theorem}
 
\newtheorem*{mainthm*}{Main Theorem}
\newtheorem{lem}[thm]{Lemma}
\newtheorem{prop}[thm]{Proposition}
\newtheorem{quest}[thm]{Question}

\theoremstyle{remark}
\newtheorem{rem}[thm]{Remark}
\newtheorem{ex}[thm]{Example}

\theoremstyle{definition}
\newtheorem{defn}[thm]{Definition}

\newcommand{\R}{\mathbb{R}}
\newcommand{\C}{\mathbb{C}}

\DeclareMathOperator{\dvol}{dvol}
\DeclareMathOperator{\im}{im}
\DeclareMathOperator{\vol}{vol}

\begin{document}
\bibliographystyle{amsalpha}
\maketitle
	
\begin{abstract}
We introduce the notion of weakly Kähler hyperbolic manifold which generalizes that of Kähler hyperbolic manifold given in the early '90s by M. Gromov, and establish its basic features. We then investigate its spectral properties and show a spectral gap result (on a suitable modification). 

As applications, we prove that weakly Kähler hyperbolic manifolds are of general type and we study the geometry of their subvarieties and entire curves, verifying ---among other things--- various aspects of the Lang and the Green--Griffiths conjectures for this class of manifolds.
\end{abstract}
	
\section{Introduction}
Ran the year 1991 when M. Gromov introduced, in one of his seminal papers \cite{Gro91}, the concept of Kähler hyperbolic manifold. These are compact Kähler manifolds endowed with a Kähler form whose pull-back to the universal cover not only is $d$-exact, but also has a primitive which is bounded. Gromov showed in that paper several beautiful and highly non trivial properties for such a manifold $X$, among which we find in order: the vanishing of $L^2$-Hodge numbers $h^{p,q}_{(2)}(X)$ for $p+q\ne\dim X$ and their non-vanishing for $p+q=\dim X$, the consequent non-vanishing for the Euler characteristic of the sheaves of holomorphic $p$-forms as well as for the topological Euler characteristic, the bigness of their canonical bundle (and hence their projectivity, by Moishezon's theorem), and stated their Kobayashi hyperbolicity.

Since then, a quite huge amount of literature has been written in connection and around the notion of Kähler hyperbolicity; just to cite very few, we find \cite{Kol95,Eys97,Hit00,JZ00,McM00,CX01,McN02,BKT13,CY18,Li19,MP22}. 

Here we focus on a variation around the notion of Kähler hyperbolicity, especially in connection with Kobayashi hyperbolicity and distribution of entire curves, which naturally arose in trying to answer the following question raised in \cite{BD21}. Kähler hyperbolic manifolds form a remarkable class of Kobayashi hyperbolic projective manifolds, and Lang's conjecture predicts that these should be of general type together with all of their subvarieties (regardless of whether they are singular or not). Given that general typeness of Kähler hyperbolic manifolds and their \emph{smooth} subvarieties is already contained in \cite{Gro91}, does it hold true for singular subvarieties as well? 

Answering this question requires to prove the vanishing/non-vanishing of the $L^2$-Hodge numbers as above, but this time for a desingularization of a singular subvariety of a Kähler hyperbolic manifold. In this case the Kähler hyperbolic metric is pulled-back and so it is no longer positive everywhere, but it degenerates along some exceptional locus (cf. Definition \ref{def:semiK} of semi-Kähler hyperbolic manifold, introduced in \cite{Eys97}). Especially, we loose completeness, which is particularly annoying in $L^2$-Hodge theory.

We introduce, more generally, the notion of weakly Kähler hyperbolic manifold (cf. Definition \ref{def:weaklyK}, where we allow the cohomology class to be merely big and nef instead of Kähler). To some extent, the definition of Kähler hyperbolicity can be seen as an incarnation of negative (holomorphic sectional) curvature. Thus, morally, a weakly Kähler hyperbolic manifold should be thought as \lq\lq non-positively curved\rq\rq{} and \lq\lq negatively curved\rq\rq{} in a Zariski dense open set which corresponds to the locus where the big and nef cohomology class is Kähler (in a precise sense, cf. Subsection \ref{subsect:nonKnull}). 

It turns out that we still can prove the spectral gap, vanishing and non-vanishing theorems in this general situation (cf. (i) of Theorem \ref{thm:spectralgap} and (ii) of Theorem \ref{thm:spectralgap}), but the price to pay is that these are obtained after passing to a suitable modification of the given weakly Kähler hyperbolic manifold. 

We collect these results in the following statement.

\begin{bigthm}[Spectral gap for weakly Kähler hyperbolic manifolds]
Let $M$ be a weakly Kähler hyperbolic manifold of complex dimension $m$. 

Then, there exists a modification $\nu\colon M'\to M$ and a Kähler form $\omega$ on $M'$ such that $0$ is not in the spectrum of the $L^2$ $\bar\partial$-Laplacian $\Delta_{\bar\partial,p,0}$ of the Kähler universal cover of $(M',\omega)$ for all $0\le p\le m-1$, and moreover $\Delta_{\bar\partial,m,0}$ has closed range and nontrivial (infinite dimensional indeed) kernel. 
\end{bigthm}

While the general strategy to establish such a spectral gap result dates back to the works \cite{Gro91,Eys97}, several quite recent and deep technologies in K\"ahler geometry are needed to achieve the proof, such as the theory of complex Monge--Ampère equations in big cohomology classes developed in \cite{BEGZ10}. 

All in all, we will see that this gives in particular that a weakly Kähler hyperbolic manifold admits a modification which is of general type, and hence it is of general type itself. Thus, this provides a confirmation of Lang's conjecture for Kähler hyperbolic manifolds (cf. Theorem \ref{thm:langkhyp}).

\begin{bigthm}[Lang's conjecture for Kähler hyperbolic manifolds]
Let $M$ be a K\"ahler hyperbolic manifold. 

Then, $M$ is Kobayashi hyperbolic and given any $X\subseteq M$, a possibly singular closed subvariety, we have that any desingularization $\hat X\to X$ is of general type, \textsl{i.e.} $X$ is of general type.
\end{bigthm} 

We saw that weakly Kähler hyperbolic manifolds are of general type. On the one hand, as such, the Green--Griffiths conjecture predicts that there should exists a proper subvariety containing the image of every entire curve. On the other hand, if a compact Hermitian manifold of non-positive holomorphic sectional curvature admits an entire curve (with bounded derivative), then it must totally geodesically land following the directions where the curvature is zero \cite{Kob78}. Therefore, pursuing the philosophical parallel between weakly Kähler hyperbolicity and (semi)negative holomorphic sectional curvature, we show that entire curves in a weakly Kähler hyperbolic manifold have to sit in the non-Kähler locus of the cohomology classes which provide weakly Kähler hyperbolicity, which is always a proper subvariety, see Theorem \ref{thm:GG}.

\begin{bigthm}[Green--Griffiths' conjecture for weakly Kähler hyperbolic manifolds]
Let $M$ be a weakly Kähler hyperbolic manifold. Then, there exists a proper subvariety $Z\subsetneq M$ which contains the image of all entire curves traced in $M$. 
\end{bigthm} 

This locus thus morally encodes the flat directions in our manifold, and the algebraic degeneracy of entire curves is proved by coupling Ahlfors' currents with a sort of linear isoperimetric inequality which holomorphic discs must satisfy thanks to the weakly Kähler hyperbolicity assumption.

\smallskip

Let us finally point out that, even if we are unfortunately not able for the moment to produce examples of weakly Kähler hyperbolic manifolds which are not semi-Kähler hyperbolic (see Remark \ref{rem:weaklynotsemi}), Theorems A and C are completely new even in the semi-K\"ahler hyperbolic setting (except for Theorem A in dimension 2 \cite{Eys97}), which is already a broad and very natural generalization of Kähler hyperbolicity. Moreover, the semi-Kähler hyperbolic setting is the minimal one needed in order to treat singular subvarieties of Kähler hyperbolic manifolds, and hence to prove Theorem B. 

Last but not least, our definition of weakly Kähler hyperbolic manifolds seems to be the most general possible in order to make the so-called Vafa--Witten trick work to get the non-vanishing in $L^2$-cohomology, and this is one reason why we believe it is the right one to consider in this context.

\subsubsection*{Acknowledgements} We warmly thank Benoît Claudon for pointing out the reference \cite{ABR92}.

\section{Background material, basic definitions and properties, and examples}

In this section we first of all recall the classical notion of Kähler hyperbolicity, and then consider a new weaker notion which we shall call weakly Kähler hyperbolicity which is better suited and perhaps more natural to consider from the point of view of birational geometry.

We shall provide examples and basic properties, and then recall some basic facts in spectral theory, Kobayashi hyperbolicity, special loci of cohomology classes on compact Kähler manifolds, Ahlfors currents, and pluripotential theory in order to make more reader friendly and self-contained the rest of the paper.

\subsection{K\"ahler hyperbolic and weakly Kähler hyperbolic manifolds: definitions and first properties}

We begin by recalling the original definition of Kähler hyperbolic manifolds due to Gromov. Let $(M,\omega)$ be a compact connected K\"ahler manifold, and let $\pi\colon(\tilde M,\tilde\omega)\to (M,\omega)$ be its K\"ahler universal cover, so that $\pi^*\omega:=\tilde\omega$. 

\begin{defn}[\cite{Gro91}]
The manifold $M$ is said to be \emph{K\"ahler hyperbolic} if it carries a Kähler metric $\omega$ such that there exists on $\tilde M$ a $1$-form $\alpha$ in such a way that $\tilde\omega=d\alpha$ and $|\alpha|_{\tilde\omega}$ is bounded on $\tilde M$.
\end{defn}

A natural corresponding notion in the framework of birational geometry was considered in \cite{Eys97} (see also \cite{Kol95}), and is the following.

\begin{defn}\label{def:semiK}
Given a real closed $(1,1)$-form $\mu$ on $M$ such that $\mu$ is non-negative and moreover strictly positive on a nonempty Zariski open set, one says that $(M,\mu)$ is \emph{semi-K\"ahler hyperbolic} if there exists a $1$-form $\alpha$ on $\tilde M$ such that $\tilde\mu:=\pi^*\mu=d\alpha$ and $|\alpha|_{\tilde\omega}$ is bounded on $\tilde M$.
\end{defn} 

Note that the boundedness of the primitive does not depend on the particular K\"ahler metric $\omega$ that we fix on $M$ since they are all quasi-isometric. 

More generally, we introduce the following variant:

\begin{defn}\label{def:weaklyK}
A compact K\"ahler manifold $M$ is \emph{weakly Kähler hyperbolic} if there exists a closed real $(1,1)$-form $\mu$ whose cohomology class $[\mu]\in H^{1,1}(M,\mathbb R)$ is big and nef and such that there exists on $\tilde M$ a $1$-form $\alpha$ with $\tilde\mu:=\pi^*\mu=d\alpha$ and $|\alpha|_{\tilde\omega}$ is bounded on $\tilde M$. 
\end{defn}

We shall recall in Subsection \ref{subsect:nonKnull} the notion of big and nef cohomology classes.

\begin{rem}\label{rem:khypcohomological}
These three notions depend only on the cohomology class of the considered $(1,1)$-form, for any other representative differs by an exact $1$-form on $M$, whose pull-back to $\tilde M$ is bounded in $L^\infty$ norm, since $M$ compact.
\end{rem}

Next, given a compact Kähler manifold $M$, call $\mathcal{W}_M$ the positive convex cone of all big and nef cohomology classes $[\mu]\in H^{1,1}(X,\mathbb R)$ which can be represented by a smooth form $\mu$ as in Definition \ref{def:weaklyK}. Sometimes, by a slight abuse of notation, we shall say that a smooth closed real $(1,1)$-form $\mu$ belongs to $\mathcal W_M$ if its cohomology class $[\mu]$ does. Clearly, such a $M$ is weakly Kähler hyperbolic if and only if $\mathcal W_M\ne\emptyset$.

\begin{rem}\label{rem:khyphereditary}
The property of being Kähler hyperbolic is easily seen to be inherited by all smooth closed irreducible submanifolds. More generally, the properties of being semi-K\"ahler hyperbolic and weakly Kähler hyperbolic are inherited by smooth closed irreducible submanifolds too, provided 
\begin{itemize}
\item they intersect non trivially the locus where $\mu$ is positive definite in the former case,
\item the top power of the restriction to the given subvariety of some class in $\mathcal W_M$ is strictly positive in the latter case. This is because the restriction of a nef class stays nef, and a nef class is big if and only if the integral of its top power is positive (cf. Subsection \ref{subsect:nonKnull}).
\end{itemize}
\end{rem}

Next, the above notions of Kähler hyperbolicity are preserved under finite étale coverings. More precisely we have.

\begin{prop}\label{prop:khypupdown}
Let $\nu\colon\hat M\to M$ be a finite étale cover. Then $M$ is Kähler hyperbolic (resp. weakly Kähler hyperbolic) if and only if $\hat M$ is so. 
\end{prop}

\begin{proof}
If $M$ satisfies one of the above Kähler hyperbolicity properties, it is immediate by pulling-back via $\nu$ that $\hat M$ satisfies the same property. 

Conversely, let $H\le\pi_1(M)$ be the subgroup of finite index which corresponds to the covering $\nu\colon\hat M\to M$. The group $\pi_1(M)$ acts on the set of, say, left cosets $\pi_1(M)/H$ by permutation. Then, we have a group homomorphism $\pi_1(M)\to\mathfrak S(\pi_1(M)/H)$ from $\pi_1(M)$ to the finite group of permutations of $\pi_1(M)/H$. It is straightforward to verify that its kernel $N$ is a normal subgroup of $\pi_1(M)$ of finite index and contained in $H$. Corresponding to $N$ we thus have a finite Galois étale cover $\rho\colon\check M\to M$ which factorizes through $\hat M$. By the first part, $\check M$ satisfies the same Kähler hyperbolicity properties of $\hat M$. But now, the average under the deck transformation group of $\rho\colon\check M\to M$ of any form $\mu\in\mathcal W_{\check M}$ gives a non zero invariant form which descends to a form in $\mathcal W_M$. It is immediate to see that this form is Kähler if $\mu$ is, and we are done.
\end{proof}

The mere fact that the Kähler form becomes exact after pulling-back to the universal cover implies that the manifold has large fundamental group in the sense of Kollár \cite[\S 4.1]{Kol95}. 
We will see later on that in the weaker context of weakly Kähler hyperbolicity one can say that $M$ has generically large fundamental group (still in the sense of Kollár \textsl{loc. cit.}). 

Here is another remarkable property of the fundamental group of a weakly Kähler hyperbolic manifold.

\begin{prop}\label{prop:nonamenable}
The fundamental group of a weakly K\"ahler hyperbolic manifold $M$ is not amenable.

In particular, it has exponential growth for any finite system of generators and can be neither virtually abelian nor virtually nilpotent. 
\end{prop}

\begin{proof}
If the fundamental group of $M$ were amenable, then the pull-back morphism $H^2(M)\to H^2_\beta(\tilde M)$, where the $H^i_\beta(\tilde M)$'s are the de Rham cohomology groups based on differential forms $\alpha$ such that $\alpha$ and $d\alpha$ are uniformly bounded, would be injective thanks to \cite[Proposition]{ABW92}. But this is impossibile since any $\mu\in\mathcal W_M$ gives a non zero class in $H^2(M)$ which pulls-back to zero in $H^2_\beta(\tilde M)$.
\end{proof}

Weakly Kähler hyperbolic manifolds do satisfy the following linear isoperimetric inequality (cf. the proof of Corollary \ref{cor:khypkobhyp} for an analogue inequality for holomorphic discs in the stronger Kähler hyperbolic setting).

\begin{prop}
Let $M$ be a weakly Kähler hyperbolic manifold, $g$ any Riemannian metric on $M$, and consider the Riemannian universal cover $\pi\colon (\tilde M,\pi^*g)\to (M,g)$. 

Then, there exists a constant $C>0$ such that for any bounded domain with $C^1$ boundary $\Omega\subset\tilde M$ we have the following linear isoperimetric inequality
$$
\mathrm{Volume}_{\pi^*g}(\Omega) \le C\, \mathrm{Area}_{\pi^*g} ( \partial \Omega).
$$
\end{prop}

\begin{proof}
Let $\mu\in\mathcal W_M$ and let $\alpha$ be a bounded real $1$-form such that $\pi^*\mu=d\alpha$. Then, the $(2n-1)$-form $\beta:=\alpha \wedge\pi^*\mu^{n-1}$ satisfies $d\beta=\pi^*\mu^n$, and it is bounded. 

Given any Riemannian metric $g$ on $M$, normalize it in such a way that
$$
\int_M\dvol_g=\int_M\mu^n.
$$
Since $\int_M\colon H^{2n}(M, \R) \to \R$ is an isomorphism, it follows that there exists a smooth $(2n-1)$-form $\eta$ on $M$ such that $\dvol_g-\mu^n=d\eta$. Therefore, $\pi^*\dvol_g= d(\pi^* \eta + \beta)$ has a bounded primitive, say 
$$
|\pi^* \eta + \beta|_{\pi^*g}\le C.
$$ 
This means that for any $v_1,\dots,v_{2n-1}\in T_{\tilde M}$, one has
$$
|(\pi^* \eta + \beta)(v_1\wedge\cdots\wedge v_{2n-1})|\le C\, |v_1\wedge\cdots\wedge v_{2n-1}|_{\pi^*g},
$$ 
and therefore, it follows using Stokes' formula that for every relatively compact domain with $C^1$-boundary $\Omega \subset \tilde{M}$ one has
$$
\mathrm{Volume}_{\pi^*g}(\Omega) \le C\, \mathrm{Area}_{\pi^*g} ( \partial \Omega).
$$ 
\end{proof}

As a corollary, we immediately get a spectral gap type result for the Laplace--Beltrami operator (with respect to \emph{any} Riemannian metric coming from the base) on the universal cover of a weakly Kählerähler hyperbolic manifold (see Theorem \ref{thm:spectralgap} for more spectral properties of weakly Kähler hyperbolic manifolds).

\begin{cor}\label{cor:zeronotspect}
Let $M$ be a weakly Kähler hyperbolic manifold, with universal cover $\pi\colon\tilde M\to M$, and let $g$ be any Riemannian metric on $M$. 
Then, zero is not in the spectrum of the Laplace--Beltrami operator of $(\tilde M, \pi^*g)$.
\end{cor}
\begin{proof}
Replacing a Riemannian metric $g$ by a positive multiple $c\,g$, has the effect of rescaling the corresponding Laplace--Beltrami operator by $1/c$. Thus, we can suppose that $g$ is normalized as in the proposition above, since the rescaling does not change the presence of zero in the spectrum. The corollary then follows from Cheeger's inequality \cite{Che70} (see also \cite[Theorem VI.1.2]{Cha01}).  
\end{proof}

\subsubsection{Examples}
Let us mention here several examples to clarify which kind of objects we are concretely working with.

\begin{ex}[Dimension one]\label{ex:dimone}
For $X$ a compact Riemann surface, weakly Kähler hyperbolicity is equivalent to Kähler hyperbolicity, and they are in turn both equivalent to the classical hyperbolicity of $X$. 

Indeed, if $(X,\mu)$ is weakly Kähler hyperbolic, then $\int_X\mu>0$ and thus there exists a Kähler form $\omega$ cohomologous to $\mu$. Therefore, $(X,\omega)$ is Kähler hyperbolic by Remark \ref{rem:khypcohomological}. 

On the other hand, Kähler hyperbolic implies Kobayashi hyperbolic (see Corollary \ref{cor:khypkobhyp}), which in turns implies negative curvature, which give us back Kähler hyperbolicity (see next example).
\end{ex}

\begin{ex}[K\"ahler hyperbolic manifolds]
Examples of K\"ahler hyperbolic manifolds are given for instance by 
\begin{itemize}
\item compact K\"ahler manifolds homotopically equivalent to a Riemannian manifold with negative Riemannian sectional curvature \cite[Theorem 3.1 and Lemma 3.2]{CY18};
\item compact Hermitian locally symmetric spaces of non compact type \cite[Proposition 8.6 and subsequent page]{Bal06};  
\item smooth closed submanifolds of K\"ahler hyperbolic manifolds, by Remark \ref{rem:khyphereditary};
\item products of K\"ahler hyperbolic manifolds.
\item compact K\"ahler manifolds admitting a finite map to a K\"ahler manifold $Y$ carrying a K\"ahler metric with a bounded primitive on the universal covering space (cf. Lemma \ref{bounded}). This includes quotients of bounded symmetric domains by non necessarily cocompact torsion free discrete subgroups.
\item compact K\"ahler manifolds admitting a finite map to $\mathcal{A}_{g,P}$ the moduli stack of genus $g$ abelian varieties with a given polarization type $P$ or to $\mathcal{M}_{g,n}$ \cite{McM00} (see the remark below). 
\end{itemize}
\end{ex}

\begin{rem}
There is also a non compact version of K\"ahler hyperbolicity considered for instance in \cite{McM00}, where it is shown that the moduli space of Riemann surfaces is K\"ahler hyperbolic. We hope to reconsider this further notion in a future paper, especially in connection with the logarithmic version of Lang's conjecture.
\end{rem}

\begin{ex}[Semi-K\"ahler hyperbolic manifolds.]
The main examples we have in mind are given by compact K\"ahler manifolds admitting a generically finite morphism to a K\"ahler hyperbolic manifold (see Proposition \ref{prop:genfinite}). In particular, explicit examples are given by
\begin{itemize}
\item the blow-up of a K\"ahler hyperbolic manifold along a smooth center;
\item any resolution of singularities of a (singular) subvariety of a K\"ahler hyperbolic manifold.
\item compact K\"ahler manifolds admitting a generically finite map to a K\"ahler manifold $Y$ carrying a K\"ahler metric with a bounded primitive on the universal covering space (cf. again Lemma \ref{bounded}). As before, this includes quotients of bounded symmetric domains by non necessarily cocompact torsion free discrete subgroups.
\item compact K\"ahler manifolds admitting a generically finite map to $\mathcal{A}_{g,P}$ the moduli stack of genus $g$ abelian varieties with a given polarization type $P$ or to $\mathcal{M}_{g,n}$.
\end{itemize}
Observe that manifolds in the first explicit class above are semi-K\"ahler hyperbolic but \emph{never} K\"ahler hyperbolic as soon as their dimension is greater than one, since they always contain rational curves in the exceptional divisor.  
\end{ex}

One of our main results will be to show that a weakly Kähler hyperbolic manifold is of general type (and hence projective, being both Moishezon and Kähler). Hence any compact complex manifold whose Kodaira dimension is not maximal gives an example of something which is \emph{not} weakly Kähler hyperbolic.

\begin{rem}\label{rem:weaklynotsemi}
Unfortunately, we do not dispose at the moment of examples of weakly K\"ahler hyperbolic manifolds which are not semi--K\"ahler hyperbolic. These would be urgently needed, if any.
\end{rem}

Summing up, the picture is:
$$
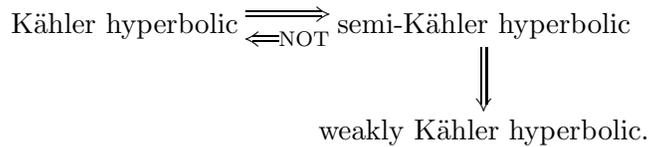
\xymatrix{
\textrm{Kähler hyperbolic} \ar@{=>}[r]<1ex> & \textrm{semi-Kähler hyperbolic} \ar@{=>}[d] \ar@{=>}[l]<1ex>|{\textrm{NOT}} \\
& \textrm{weakly Kähler hyperbolic.} \\
}
$$

\subsection{Basics on spectral theory and Atiyah $L^2$ theory}

We start this section by recalling some basic spectral theory of the Hodge--Kodaira Laplacian. 
Let $(N,h)$ be a complex manifold of complex dimension $n$ endowed with a complete Hermitian metric $h$. We denote with $\overline{\partial}_{p,q}\colon \Omega^{p,q}(N)\rightarrow \Omega^{p,q+1}(N)$ the Dolbeault operator acting on $(p,q)$-forms  and with   
 $\overline{\partial}_{p,q}^t:\Omega^{p,q+1}(N)\rightarrow \Omega^{p,q}(N)$  its  formal adjoint with respect to the metric $h$.  The Hodge--Kodaira Laplacian acting on $(p,q)$-forms, denoted here with $\Delta_{\overline{\partial},p,q}:\Omega^{p,q}(N)\rightarrow \Omega^{p,q}(N)$, is defined as 
 $$
 \Delta_{\overline{\partial},p,q}:= \overline{\partial}_{p,q}^t\circ \overline{\partial}_{p,q}+\overline{\partial}_{p,q-1}\circ \overline{\partial}_{p,q-1}^t.
 $$ 
 Let $L^2\Omega^{p,q}(N,h)$ be the Hilbert space of $(p,q)$-forms with measurable coefficients $u$ such that 
 $$
 \langle\!\langle u,u\rangle\!\rangle_{L^2\Omega^{p,q}(N,h)}:=\int_N |u|^2_h\,\dvol_h<\infty,
 $$ 
 where $\dvol_h$ stands for the volume form corresponding to $h$ and, with a little abuse of notations, we have still used $h$ to denote the Hermitian metric induced by $h$ on $\Lambda^{p,q}(N)$. It is well known that $\Omega_c^{p,q}(N)$, the space of smooth $(p,q)$-forms with compact support, is dense in $L^2\Omega^{p,q}(N,h)$. 

We now look at 
\begin{equation}
\label{terzo}
\Delta_{\overline{\partial},p,q}\colon L^2\Omega^{p,q}(N,h)\rightarrow L^2\Omega^{p,q}(N,h)
\end{equation}
as an unbounded, closable and densely defined operator defined on $\Omega^{p,q}_c(N)$. A  well known result says that \eqref{terzo} is essentially self-adjoint, see e.g. \cite[Prop. 12.2]{BDIP02}.  Since \eqref{terzo} is formally self-adjoint this is equivalent to saying that \eqref{terzo} admits a unique closed extension. Henceforth with a little abuse of notation we still denote with
\begin{equation}
\label{sesto}
\Delta_{\overline{\partial},p,q}\colon L^2\Omega^{p,q}(N,h)\rightarrow L^2\Omega^{p,q}(N,h)
\end{equation}
the unique closed (and hence self-adjoint) extension of \eqref{terzo}.  
Since $(N,h)$ is complete \eqref{sesto} is nothing but
 $$
 \overline{\partial}_{p,q}^*\circ \overline{\partial}_{p,q}+\overline{\partial}_{p,q-1}\circ \overline{\partial}_{p,q-1}^*\colon L^2\Omega^{p,q}(N,h)\to L^2\Omega^{p,q}(N,h).
 $$ 
Here 
$$
\overline{\partial}_{p,q}\colon L^2\Omega^{p,q}(N,h)\to L^2\Omega^{p,q+1}(N,h)
$$ 
and 
$$
\overline{\partial}_{p,q-1}\colon L^2\Omega^{p,q-1}(N,h)\to L^2\Omega^{p,q}(N,h)
$$ 
are the unique closed extensions of 
$$
\overline{\partial}_{p,q}\colon\Omega_c^{p,q}(N)\to \Omega_c^{p,q+1}(N),
$$
and 
$$
\overline{\partial}_{p,q-1}\colon\Omega_c^{p,q-1}(N)\rightarrow \Omega_c^{p,q}(N)
$$ 
respectively, and  
$$
\overline{\partial}_{p,q}^*\colon L^2\Omega^{p,q+1}(N,h)\to L^2\Omega^{p,q}(N,h)
$$ 
and 
$$
\overline{\partial}_{p,q-1}^*\colon L^2\Omega^{p,q}(N,h)\to L^2\Omega^{p,q-1}(N,h)
$$  
are the corresponding adjoints. Note that since the manifold is assumed to be complete, $\overline{\partial}_{p,q}^*$ is the unique closed extension of the formal adjoint $\overline{\partial}_{p,q}^t$. It follows immediately that in $L^2\Omega^{p,q}(N,h)$ we have
$$
\ker(\Delta_{\overline{\partial},p,q})=\ker(\overline{\partial}_{p,q})\cap \ker(\overline{\partial}_{p,q-1}^*),\quad \overline{\operatorname{im}(\Delta_{\overline{\partial},p,q})}=\overline{\operatorname{im}(\overline{\partial}_{p,q-1})}\oplus \overline{\operatorname{im}(\overline{\partial}^*_{p,q})},
$$
and the following $L^2$-Hodge--Kodaira orthogonal decomposition holds true
$$
L^2\Omega^{p,q}(N,h)\simeq\ker(\Delta_{\overline{\partial},p,q})\oplus\overline{\operatorname{im}(\overline{\partial}_{p,q-1})}\oplus \overline{\operatorname{im}(\overline{\partial}^*_{p,q})}.
$$

Now we continue this section by recalling the definition and the basic properties of the $L^2$-Hodge numbers. We introduce only what is strictly necessary for our scopes with no goal of completeness.  We invite the reader to consult the seminal paper of Atiyah \cite{Ati76} for more on this. Moreover, for an in-depth treatment of this topic we refer to \cite{Luc02} whereas quick introductions can be found in \cite{Kol95}, \cite{MM07} and \cite{Roe98}. 

Let $(M,h)$ be a compact Hermitian manifold of complex dimension $m$. Let $\pi\colon \tilde{M}\rightarrow M$ be a Galois covering of $M$ and let $\tilde{h}:=\pi^*h$ be the pull-back metric on $\tilde{M}$. Let $\Gamma$ be the group of deck transformations acting on $\tilde{M}$, $\Gamma\times \tilde{M}\rightarrow \tilde{M}$. We recall that  $\Gamma$ is a discrete group acting fiberwise on $\tilde{M}$ and such that the action is transitive on each fiber and properly discontinuous. In particular we have $\tilde{M}/\Gamma=M$. 
An open subset $A\subset \tilde{M}$ is a  fundamental domain of the action of $\Gamma$ on $\tilde{M}$ if
\begin{itemize}
\item $\tilde{M}=\bigcup_{\gamma\in \Gamma}\gamma(\overline{A})$, 
\item $\gamma_1(A)\cap\gamma_2(A)=\emptyset$ for every $\gamma_1, \gamma_2\in \Gamma$ with $\gamma_1\neq \gamma_2$,
\item $\overline{A}\setminus A$ has zero measure.
\end{itemize}   
It is not difficult to show that, in the sense of tensor product of Hilbert spaces, it holds
$$
L^2\Omega^{p,q}(\tilde{M},\tilde{h})\simeq L^2\Gamma \otimes L^2\Omega^{p,q}(A,\tilde{h}|_A)\cong L^2\Gamma\otimes L^2\Omega^{p,q}(M,h).
$$ 
Moreover, it is clear that $\Gamma$ acts on $L^2\Omega^{p,q}(\tilde{M},\tilde{h})$ by isometries. Let us consider a $\Gamma$-module $V\subset L^2\Omega^{p,q}(\tilde{M},\tilde{h})$, that is a closed subspace of $L^2\Omega^{p,q}(\tilde{M},\tilde{h})$ which is invariant under the action of $\Gamma$. If $\{\eta_j\}_{j\in \mathbb{N}}$ is an orthonormal basis for $V$, then the following (possibly infinite) real number
$$
\sum_{j\in \mathbb{N}}\int_A|\eta_j|^2_{\tilde{h}}\,\dvol_{\tilde{h}}
$$ 
is well-defined, that is it does depend neither upon the choice of the orthonormal basis of $V$ nor on the choice of the fundamental domain of the action of $\Gamma$ on $\tilde{M}$, see \textsl{e.g.} \cite{Ati76}. The von Neumann dimension of a $\Gamma$-module $V$ is thus defined as 
$$
\dim_{\Gamma}(V):=\sum_{j\in \mathbb{N}}\int_A |\eta_j|^2_{\tilde{h}}\,\dvol_{\tilde{h}},
$$ 
where $\{\eta_j\}_{j\in \mathbb{N}}$ is any orthonormal basis for $V$ and $A$ is any fundamental domain for the action of $\Gamma$ on $\tilde{M}$.  

Since the Hodge--Kodaira Laplacian $\Delta_{\overline{\partial},p,q}\colon L^2\Omega^{p,q}(\tilde{M},\tilde{h})\rightarrow L^2\Omega^{p,q}(\tilde{M},\tilde{h})$ commutes with the action of $\Gamma$, a natural and important example of $\Gamma$-module with \emph{finite} $\Gamma$-dimension is provided by 
$\ker\bigl(\Delta_{\overline{\partial},p,q}\bigr)$, the space of $L^2$-harmonic $(p,q)$-forms on $(\tilde{M},\tilde{h})$, for each $p,q=0,\dots,m$, cf. \cite{Ati76}.  The $L^2$-Hodge numbers of  $(M,h)$ with respect to a Galois $\Gamma$-covering $\pi\colon \tilde{M}\rightarrow M$ endowed with the pull-back metric $\tilde{h}:=\pi^*h$ are then defined as 
$$
h^{p,q}_{(2),\Gamma,\overline{\partial}}(M):=\dim_{\Gamma}\bigl(\ker(\Delta_{\overline{\partial},p,q})\bigr).
$$
We point out that $h^{p,q}_{(2),\Gamma,\overline{\partial}}(M)$ are nonnegative real numbers independent on the choice of the Hermitian metric $h$ on $M$. We recall now a deep result which is nothing but the celebrated Atiyah $L^2$-index theorem applied to a compact complex manifold (see \cite{Ati76}). 
\begin{thm}
\label{thm:gammaHodge}
Let $(M,h)$ be a compact complex Hermitian manifold of complex dimension $m$. Let $\pi:\tilde{M}\rightarrow M$ be a Galois $\Gamma$-covering of $M$ and let $\tilde{h}:=\pi^{*}h$ be the pull-back metric on $\tilde{M}$. Then for any fixed $p\in \{0,...,m\}$ we have 
$$
\sum_{q=0}^m(-1)^q\,h_{\overline{\partial}}^{p,q}(M)=\sum_{q=0}^m(-1)^q\,h_{(2),\Gamma,\overline{\partial}}^{p,q}(M).
$$
\end{thm}
Finally we conclude this section with a remark about the notation. In the case of a compact K\"ahler manifold $(M,h)$ we will simply denote with $h^{p,q}_{(2)}(M)$ the $L^2$-Hodge numbers of $M$ with respect to its universal covering.

\subsection{Kobayashi hyperbolicity, manifolds of general type and Lang's conjecture}
We refer the reader to \cite{Kob98} for a omni-comprehensive account of Kobayashi hyperbolicity. Let $X$ be a complex manifold, or more generally a complex space. One can define a pseudodistance $d_X$ on $X$ as the largest among those pseudodistances $d$ with the property $\rho(a,b)\ge d(f(a),f(b))$, where $\rho$ is the Poincaré distance on the complex unit disc $\mathbb D$, for all holomorphic map $f\colon\mathbb D\to X$, and all $a,b\in\mathbb D$. 

\begin{defn}
The pseudodistance $d_X$ is called \emph{Kobayashi pseudodistance}, and the manifold $X$ is said to be \emph{Kobayashi hyperbolic} if $d_X$ is a genuine distance, \textsl{i.e.} it separates points. 

More generally, given a closed (in the Euclidean topology) subset $F\subseteq X$, we say that $X$ is \emph{Kobayashi hyperbolic modulo $F$} if for every pair of distinct points $p,q$ of $X$ we have $d_X(p,q)>0$ unless both are contained in $F$.
\end{defn}

From the very definition, we immediately see that being Kobayashi hyperbolic is an hereditary property for complex subspaces of a Kobayashi hyperbolic complex space. When $X$ is moreover compact, Kobayashi hyperbolicity turns out to be equivalent to Brody hyperbolicity, namely $X$ is Kobayashi hyperbolic if and only if it does not admit any entire curve, \textsl{i.e.} a \emph{non constant} holomorphic map $f\colon\mathbb C\to X$. More generally, if $X$ is hyperbolic modulo $F$, then any entire curve must land into $F$; conversely, if $X$ is compact then for each $x\in X$ we have that 
$$
F_x:=\{y\in X\mid d_X(x,y)=0\}
$$ 
is a connected set and, if $F_x\ne\{x\}$, then there exists an entire curve whose image is inside $F_x$.

When $X$ is a smooth projective manifold (or, more generally, a compact Kähler manifold) the analytic property of being Kobayashi hyperbolic is expected to be completely characterized by the following algebraic property: $X$ is of general type together with all of its subvarieties. This is the content of Lang's conjecture \cite[Conjecture 5.6]{Lan86}. 

Recall that a smooth compact complex manifold $X$ is said to be of \emph{general type} if its canonical bundle $K_X$ is \emph{big}, \textsl{i.e.} the growth of the dimension of the space of pluricanonical sections is the largest possible $h^0(X,K_X^{\otimes m})\sim m^{\dim X}$. When $X$ is singular, by definition, it is of general type if it admits a desingularization of general type. 

The direction of Lang's conjecture which state that Kobayashi hyperbolic manifolds are of general type together with their subvarieties has been verified in some particular cases besides the trivial case of curves, among which we mention
\begin{enumerate}
\item compact K\"ahler surfaces (this follows from bimeromorphic classification, plus the non hyperbolicity of K3's),
\item compact K\"ahler manifolds with negative holomorphic sectional curvature \cite{WY16a,WY16b,TY17,Gue22},
\item compact, free quotients of bounded domains $\Omega\subset\C^n$ \cite{BD21,CDG20,BKT13},
\item generic projective hypersurfaces of high degree, and consequently generic projective complete intersections of high multi-degree \cite{Bro17,McQ99,DEG00,DT10,Siu04,Siu15,RY22,Cle86,Ein88,Xu94,Voi96,Pac04}.
\item Consider a compact K\"ahler manifold $X$ whose fundamental group has a semi-simple representation $\rho\colon\pi_1(X) \to\operatorname{GL}(N,\C)$ such that, for all positive dimensional irreducible subvarieties $Z\subset X$ (including $X$ itself), $\rho\bigl(\pi_1(\hat Z)\bigr)$ is an infinite group with a semi-simple real Zariski closure, where $\hat Z$ is a desingularization of $Z$. 

It follows from the ideas of \cite{Zuo96} (see \cite{Yam10} and \cite[Proposition 3.14]{CCE15}) that in this situation we can infer that $X$ is Kobayashi hyperbolic and all its subvarieties are of general type. In the case $\rho$ underlies a $\C$-VHS, these real Zariski closures are automatically semi-simple.
\end{enumerate}

The other direction can be seen as a consequence of \cite[Conjecture 5.5]{Lan86}, which Lang attributes to Green and Griffiths, which states that in a projective manifold of general type the Zariski closure of the image on an entire curve must be a proper subvariety. There is also a stronger variant of this conjecture, which predicts that in a projective manifold of general type there should exists a proper subvariety containing the image of every entire curve. Even stronger, one might ask \cite[Conjecture 4, p. 80]{Kob98} if a projective manifold of general type is hyperbolic modulo a proper algebraic subvariety, which should be the union of its subvarieties which are \emph{not} of general type. We shall see some results in this directions in Section \ref{subsect:Lang}, especially in connection with null and non-Kähler loci (cf. Subsection \ref{subsect:nonKnull}).

Next, let us mention the following beautiful characterization of Kobayashi hyperbolic compact complex manifolds.

\begin{thm}[\cite{Duv08}]
A compact complex manifold is Kobayashi hyperbolic if and only if any \lq\lq holomorphic disc\rq\rq{} in $X$ satisfies a linear isoperimetric inequality\footnote{As also mentioned in J. Duval's paper, an analogous statement was also independently proved by Bruce Kleiner in an unpublished work.} . 
\end{thm} 
This means that there exists a real constant $C>0$ such that given any holomorphic map $f\colon\mathbb D\to X$ which is say continuously differentiable up to the boundary, we have
$$
\operatorname{Area}(\mathbb D)\le C\,\operatorname{Length} (\partial\mathbb D),
$$
with respect to some (hence any, by compactness, by properly modifying the constant $C$) Hermitian metric on $X$.

Using this theorem, it is almost immediate to prove Kobayashi hyperbolicity of Kähler hyperbolic manifolds.

\begin{cor}[Cf. \cite{Gro91,CY18}]\label{cor:khypkobhyp}
Let $X$ be a Kähler hyperbolic manifold. Then, $X$ is Kobayashi hyperbolic.
\end{cor}

\begin{proof}
Let $f\colon\mathbb D\to X$ be a holomorphic disc as above and $\omega$ a Kähler form such that $\pi^*\omega=d\alpha$, where $\pi\colon\tilde X\to X$ is the universal cover and $|\alpha|_{\pi^*\omega}\le C$. Take any holomorphic lifting $\tilde f\colon\mathbb D\to\tilde X$ and observe that 
$$
\begin{aligned}
\operatorname{Area}_\omega(\mathbb D) & =\int_\mathbb D f^*\omega = \int_\mathbb D (\pi\circ\tilde f)^*\omega \\
&=\int_\mathbb D\tilde f^*d\alpha =\int_{\partial\mathbb D}\tilde f^*\alpha \\
&\le C\,\operatorname{Length}_\omega (\partial\mathbb D),
\end{aligned}
$$
since $||\alpha||_{\pi^*\omega}<C$ implies precisely that for any $v\in T_{\overline{\mathbb D}}$ one has
$$
|\tilde f^*\alpha (v)|\le C\, |\tilde f_*(v)|_{\pi^*\omega}=C\, |\pi_*\bigl(\tilde f_*(v)\bigr)|_\omega=C\, |f_*(v)|_\omega.
$$
\end{proof}

Lang's conjecture would then predict that a Kähler hyperbolic manifold $X$ should be of general type together with all of its subvarieties. One of the main theorems in \cite{Gro91} asserts that a Kähler hyperbolic manifold is indeed of general type. In particular $X$ is Moishezon and Kähler so that it is projective algebraic by Moishezon's theorem. But then, since $X$ cannot posses any rational curve, \textsl{e.g.} by Kähler hyperbolicity, its canonical bundle $K_X$ is ample (\cite[Exercise 8 p. 219]{Deb01}, see also \cite{CY18}). Also, all its closed \emph{smooth} submanifolds are of general type, being again Kähler hyperbolic (cf. Remark \ref{rem:khyphereditary}), and with ample canonical bundle. 

As announced in the introduction, in Section \ref{subsect:Lang} we shall see that also \emph{singular} subvarieties of a Kähler hyperbolic manifold are of general type, thus establishing Lang's conjecture for this class of Kobayashi hyperbolic manifolds, answering a question raised in \cite{BD21}.

\subsection{Ahlfors' currents}
Morally speaking, Ahlfors' currents provide a way to perform intersection theory with certain non compact objects, namely entire curves on compact Kähler manifolds. 

Let $f\colon\mathbb C\to (X,\omega)$ be such an entire curve, $\dim X=n$. For any smooth $(1,1)$-form $\psi\in \Omega^{1,1}(X)$, let
$$
T_{f,r}(\psi)=\int_0^r \frac{dt}t\int_{\mathbb D_t}f^*\psi,
$$
where $\mathbb D_t\subset\mathbb C$ is the disc of radius $t$. For $r>0$, consider the positive current $\Phi_r$ of bidimension $(1,1)$ given by
$$
\psi\mapsto\Phi_r(\psi)=\frac{T_{f,r}(\psi)}{T_{f,r}(\omega)}.
$$
The family $\{\Phi_r\}_{r>0}$ is bounded, and so we certainly have an accumulation point $\Phi$ which is a positive current. One important fact is that we can always choose such a limit current (which is by no means unique) to be in addition closed (see for instance \cite[Lemme 0]{Bru99}). One can see that this amounts to proving that
$$
\liminf_{r\to +\infty}\frac{S_{f,r}(\omega)}{T_{f,r}(\omega)}=0,
$$
where 
$$
S_{f,r}(\omega)=\int_0^r L_\omega(t)\frac{dt}{t},\quad L_\omega(t)=\operatorname{Lenght}_\omega\bigl(\partial\mathbb D_t\bigr).
$$

We shall call an increasing diverging sequence $\{r_k\}_{k\in\mathbb N}$ such that $\Phi_{r_k}$ converges to a closed limit an \emph{admissible sequence}. For such a sequence we certainly have $\lim_{r_k\to +\infty}S_{f,r_k}(\omega)/T_{f,r_k}(\omega)=0$. A closed positive current which arises in this way is called an \emph{Ahlfors current} associated to $f\colon\mathbb C\to (X,\omega)$. It can be thought as a \lq\lq current of integration along $f(\mathbb C)$\rq\rq{}.

Given an Ahlfors current $\Phi$, its cohomology class $[\Phi]\in H^{n-1,n-1}(X,\mathbb R)$ satisfies a basic intersection inequality which will be crucial for our applications.

\begin{prop}{\cite[Lemme 1]{Bru99}}\label{prop:ahlfors}
Let $Y\subset X$ be an irreducible closed analytic hypersurface such that $f(\mathbb C)\not\subset Y$, and let $[Y]\in H^{1,1}(X,\mathbb R)$ its Poincaré dual cohomology class. Then,
$$
[\Phi]\cdot [Y]\ge 0.
$$
\end{prop}

\subsection{Big and nef classes, non-K\"ahler locus, and null locus}\label{subsect:nonKnull}
For the terminology and main result of this section we refer the reader to \cite{Bou04}. As an excellent source for the theory of positive currents the reader may consult \cite{Dem12}.

Let $X$ be a compact Kähler manifold, and fix some reference Kähler metric $\omega$. Given a cohomology class $[\alpha]\in H^{1,1}(X,\mathbb R)$, we say that it is \emph{big} if it contains a Kähler current, \textsl{i.e.} under the natural isomorphism between cohomology of forms and currents, it can be represented by a closed positive $(1,1)$-current $T$ which is bounded below in the sense of currents by some positive multiple of $\omega$ (such a current is called a Kähler current). 

We say that $[\alpha]$ is \emph{nef} if it can be represented by smooth real $(1,1)$-forms bounded below by arbitrarily small negative multiples of $\omega$.

This terminology is borrowed from algebraic geometry: indeed if $X$ is projective and $L$ is a holomorphic line bundle on $X$, then $L$ is big (in the sense that it has maximal Kodaira--Iitaka dimension) if and only if $c_1(L)$ is big in the sense just described, and it is nef (in the sense that its intersection with any irreducible curve is non negative) if and only if $c_1(L)$ is nef as above (see \cite{Dem92b}).

Thanks to \cite[Theorem 0.5]{DP04}, a nef class $[\alpha]$ is big if and only if $\int_X\alpha^{\dim X}>0$.

\begin{ex}
A class which can be represented by a smooth closed semi-positive $(1,1)$-form $\alpha$ which is strictly positive somewhere is both big and nef. In this case, in fact, $[\alpha]$ is trivially a nef class by the very definition, and moreover $\int_X\alpha^{\dim X}>0$. Therefore, as just observed, by \cite[Theorem 0.5]{DP04}, $[\alpha]$ is a big (and nef) class. In particular a semi-Kähler hyperbolic manifold is weakly Kähler hyperbolic.
\end{ex}

Now, given a closed positive current $T$ and a positive real number $c>0$, its corresponding Lelong super-level set is the set
$$
E_c(T)=\{x\in X\mid\nu(T,x)\ge c\},
$$
where $\nu(T,x)$ is the Lelong number of $T$ at $x$. It is a fundamental theorem of Y.-T. Siu \cite{Siu74} that such sets are in fact closed analytic subvarieties of $X$. Next, call
$$
E_+(T)=\bigcup_{c>0}E_c(T),
$$
which is seen to be a countable union of analytic subvarieties.

Coming back to our big classes, we have the following definition, due to \cite{Bou04}.

\begin{defn}
Given a big class $[\alpha]\in H^{1,1}(X,\mathbb R)$, its \emph{non-Kähler locus} $E_{nK}([\alpha])$ is defined to be
$$
E_{nK}([\alpha])=\bigcap_{T\in[\alpha]}E_+(T),
$$
where $T$ ranges among all Kähler currents representing $[\alpha]$.
\end{defn}
One can prove that \cite[Theorem 3.17]{Bou04}: the class $[\alpha]$ is Kähler if and only if $E_{nK}([\alpha])=\emptyset$ (thus the non-Kähler locus gives a measure of the failure for the class $[\alpha]$ of being Kähler and also \lq\lq locates\rq\rq{} this failure), and moreover, given a big class $[\alpha]$ one can always find a Kähler current $T\in[\alpha]$ with analytic singularities such that $E_+(T)=E_{nK}([\alpha])$ (in particular, $E_{nK}([\alpha])$ is an analytic subvariety). Recall that a positive $(1,1)$-current is said to have analytic singularities if it is locally congruent to
$$
\gamma\,i\partial\bar\partial\log\sum |f_j|^2 \pmod{\textrm{smooth forms}},
$$
for some real number $\gamma>0$ and a finite set of (locally defined) holomorphic functions $f_j$, so that if $T$ has analytic singularities $E_{+}(T)$ is an analytic subset and $T$ is smooth outside of $E_{+}(T)$.

\begin{ex}
To have a more geometric flavour of what the non-Kähler locus is, suppose $X$ is projective and $L\to X$ is a holomorphic line bundle. Then, Kodaira's lemma states that $L$ is big if and only if for any ample line bundle $A\to X$ we have that $L^{\otimes k}\otimes A^{-1}$ has a non zero holomorphic global section for all $k\in\mathbb N$ sufficiently large and divisible. In other words, $L$ is big if and only if all of its sufficiently large and divisible tensor powers are isomorphic to an ample line bundle tensor an effective one. The \emph{augmented base locus} of the big line bundle $L$ is then defined as
$$
\mathbf{B}_+(L)=\bigcap_{L^{\otimes k}\sim A\otimes \mathcal O_X(E)} \operatorname{Supp}(E)
$$
where $k\in\mathbb N$ is an exponent for which Kodaira's lemma holds, $A$ is ample, $E$ is an effective divisor and $\operatorname{Supp}(E)$ is the support of the divisor $E$.

Then, if $L$ is big as a line bundle, so is $c_1(L)$ as a cohomology class, and one can prove that $E_{nK}\bigl(c_1(L)\bigr)=\mathbf B_+(L)$.
\end{ex}

Suppose now that the class $[\alpha]$ is nef. Another subset of $X$ which measures the non-Kählerity of $[\alpha]$ in this case is the so-called \emph{null locus}
$$
\operatorname{Null}([\alpha])=\bigcup_{\int_Z \alpha^{\dim Z}=0}Z,
$$ 
where $Z$ varies among all the irreducible positive dimensional analytic subvarieties of $X$ over which $\alpha$ (raised to the right power) integrates to zero. By \cite[Theorem 0.5]{DP04}, $\operatorname{Null}([\alpha])\ne X$ if and only if the class $[\alpha]$ is also big.

We have the following result, which will be an important tool for us later on.

\begin{thm}[{\cite[Theorem 1.1]{CT15}}]
If $[\alpha]\in H^{1,1}(X,\mathbb R)$ is a nef class on a compact (complex manifold bimeromorphic to a) Kähler manifold $X$, then
$$
E_{nK}([\alpha])=\operatorname{Null}([\alpha]).
$$
\end{thm}

Thus, in the nef case, the non-Kähler locus of a class $[\alpha]$ is a numerical invariant. This leads to the following.

\begin{defn}
For $M$ a weakly Kähler hyperbolic manifold, the \emph{degeneracy set} $Z_M$ is defined to be
$$
Z_M=\bigcap_{[\mu]\in\mathcal W_M} \operatorname{Null}([\mu]).
$$
\end{defn}

\begin{rem}
If $M$ is a weakly Kähler hyperbolic manifold, then there exists $\mu_0\in\mathcal W_M$ such that 
$$
Z_M=E_{nK}([\mu_0])=\operatorname{Null}([\mu_0]).
$$
In fact, by the Noetherianity, the intersection is stationary, so that $Z_M=\bigcap_{j=1}^N\operatorname{Null}([\mu_j])$, for some $\mu_1,\dots,\mu_N\in\mathcal W_M$. But then, $\mu_0:=\sum_{j=1}^N\mu_j$ belongs to $\mathcal W_M$ and 
$$
\operatorname{Null}([\mu_0])\subset\bigcap_{j=1}^N\operatorname{Null}([\mu_j]),
$$
since, for any $k=1,\dots,\dim M$, and any irreducible subvariety $Z\subseteq M$ of dimension $k$, we have $[\mu_0^k]\cdot Z \ge\sum_{j=1}^N [\mu_j^k]\cdot Z$.

In particular $Z_M$ is empty if and only if $M$ is K\"ahler hyperbolic, and if it is non empty then it cannot have isolated points.
\end{rem}

The numerical degeneracy set is always a proper (positive dimensional, whenever nonempty) subvariety. We shall see later how $Z_M$ controls some properties of geometrical nature, \textsl{e.g.} the Kobayashi hyperbolicity defect of $M$ (see Subsection \ref{subsect:Lang}), or the following.

\begin{prop}
Let $M$ be a weakly Kähler hyperbolic manifold, and $\iota\colon X\hookrightarrow M$ an irreducible closed complex analytic subvariety such that $X\not\subset Z_M$. Then, the image $\operatorname{Im}\bigl(\iota_*\colon\pi_1(X)\to\pi_1(M)\bigr)$ is infinite.
\end{prop}

Since $Z_M$ is always a proper subvariety, this exactly means that a weakly Kähler hyperbolic manifold has generically large fundamental group.

\begin{proof}
Let $\mu\in\mathcal W_M$ such that $X\not\subset\operatorname{Null}([\mu])$, so that $\int_X\mu^{\dim X}>0$. If $\operatorname{Im}\bigl(\iota_*\colon\pi_1(X)\to\pi_1(M)\bigr)$ were finite, passing to a finite cover $\nu\colon \hat X\to X$, we would have that $\iota\circ\nu\colon \hat X\to M$ lifts to the universal cover $\pi\colon\tilde M\to M$. We would thus obtain a compact subvariety $\tilde X$ of $\tilde M$ such that $\pi|_{\tilde X}\colon\tilde X\to X$ is finite. We would therefore get a contradiction, since on the one hand this would imply $\int_{\tilde X}(\pi^*\mu)^{\dim\tilde X}>0$ and on the other hand this integral has to be zero since $\pi^*\mu$ is exact. 
\end{proof}

\subsection{Monge--Ampère equation in big cohomology classes}

During the proof of Theorem \ref{thm:spectralgap} we shall need a few important results as well as the notion of non-pluripolar product, taken from \cite{BEGZ10}. We briefly recall them here for the reader's convenience. 

Given $p$ arbitrary closed positive $(1,1)$-currents $T_1,\dots, T_p$, there exists a well-defined \cite[Proposition 1.6]{BEGZ10} notion of \emph{non-pluripolar product}
$$
(T_1\wedge\cdots\wedge T_p)
$$
which gives back a closed \cite[Theorem 1.8]{BEGZ10} positive $(p,p)$-current which puts no mass on pluripolar sets. This notion enables one to give a meaning and to solve a quite general class of degenerate complex Monge--Ampère equations. In \textsl{op. cit.} it is indeed shown the following.

\begin{thm}[{\cite[Theorem A, Theorem 3.1]{BEGZ10}}]
Let $\alpha\in H^{1,1}(X,\mathbb R)$ be a big cohomology class on a compact Kähler $n$-dimensional manifold $X$. If $\lambda$ is a positive measure on $X$ which puts no mass on pluripolar sets and satisfies moreover the necessary condition $\lambda(X)=\operatorname{vol}(\alpha)$, then there exists a unique closed positive current $T\in\alpha$ such that
$$
(T^n)=\lambda.
$$
\end{thm}

Here $(T^n)$ stands for the non-pluripolar product of $T$ with itself taken $n$ times, and $\operatorname{vol}(\alpha)$ is the volume of the cohomology class $\alpha$ in the sense of \cite{Bou04}.
If moreover the measure $\lambda$ has $L^{1+\varepsilon}$ density with respect to the Lebesgue measure for some $\varepsilon>0$, then the solution $T$ has minimal singularities \cite[Theorem B, Theorem 4.1]{BEGZ10} in the sense of Demailly.

Finally, whenever the measure $\lambda$ is a smooth strictly positive volume form and the class $\alpha$ is moreover nef, then the solution $T$ is also smooth in the complement of the non-Kähler locus $E_{nK}(\alpha)$ of $\alpha$ (this is the content of \cite[Theorem C, Theorem 5.1]{BEGZ10}).

\subsection{Kähler hyperbolicity and generically finite maps}

Here we study how the different notions of Kähler hyperbolicity propagate backwards under generically finite morphisms. We first need a lemma. 

\begin{lem}\label{bounded}
Let $f\colon (M,g_M)\to (N,g_N)$ be a smooth map between compact Riemannian manifolds, and suppose that we are given a bounded $k$-form $\eta$ on the Riemannian universal cover $(\tilde N,\tilde g_N)$ of $(N,g_N)$. 

Let $\tilde f\colon (\tilde M,\tilde g_M)\to (\tilde N,\tilde g_N)$ be a lifting of $f$. Then, the pull-back $\tilde f^*\eta$ is bounded.
\end{lem}

\begin{proof}
For the pointwise norm of $\tilde{f}^*\eta$ we have at a point $p\in\tilde M$
$$
\begin{aligned}
|\tilde{f}^*\eta|_{\tilde{g}_M} & =\sup_{\beta\in \Lambda^kT_p\tilde M,\ |\beta|_{\tilde{g}_M}=1}|\tilde{f}^*\eta(\beta)| \\ 
& =\sup_{\beta\in \Lambda^kT_p\tilde M,\ |\beta|_{\tilde{g}_M}=1}|\eta\bigl(\wedge^kd_p\tilde{f}(\beta)\bigr)| \\
& \le \|\eta\|_{L^{\infty}\Omega^k(\tilde{N},\tilde{g}_N)}||\wedge^kd_p\tilde{f}||_{\mathrm{op}},
\end{aligned}
$$
where $||\wedge^kd_p\tilde{f}||_{\mathrm{op}}$ denotes the operator norm of 
$$
\wedge^kd_p\tilde{f}:(\Lambda^kT_p\tilde{M},\tilde{g}_M)\rightarrow (\Lambda^kT_p\tilde{N},\tilde{g}_N).
$$
Since $M$ is compact and the coverings are Riemannian, there exists a positive constant $C$ such that 
$$
\sup_{p\in \tilde  {M}}||\wedge^kd_p\tilde{f}||_{\mathrm{op}}\leq C.
$$ 
Thus, we can conclude that 
$$
\sup_{p\in \tilde{M}}|\tilde{f}^*\eta|_{\tilde{g}_M}\leq C\|\eta\|_{L^{\infty}\Omega^k(\tilde{N},\tilde{g}_N)},
$$ 
that is $\tilde{f}^*\eta$ is bounded. 
\end{proof}

With this lemma at our disposal, we can show the following key proposition.

\begin{prop}\label{prop:genfinite}
Let $N$ be a smooth compact connected K\"ahler manifold of dimension $\dim N=k > 0,$ and let $F\colon N \to M $ be a generically finite holomorphic map, then:  
\begin{itemize}
\item[(i)] If $M$ is K\"ahler hyperbolic, then $N$ is semi-K\"ahler hyperbolic. 
\item[(ii)] If $k  = m=\dim M$, then $N$ is  weakly K\"ahler hyperbolic if so is $M$. 
\item[(iii)] If $k < m$, $M$ is weakly K\"ahler hyperbolic, and $F(N) \not\subset Z_M$, then $N$ is weakly Kähler hyperbolic, too.
\end{itemize}
\end{prop}

\begin{proof}
Let $X = F(N)$, and consider the following diagram
$$
\xymatrix{
\tilde N \ar[rr]^{\tilde{F}}\ar[d]_p & & \tilde M\ar[dd]^\pi\\
 N \ar[dr]_F & &\\
 & X \ar@{^{(}->}[r] & M,
}
$$
where $p\colon\tilde N\to N$ is the universal covering and $\tilde{F}$ is a lifting of $F \circ p$.

Recall that the manifold $N$ is K\"ahler with respect to some K\"ahler metric $\eta$, and let $R\subset N$ be the set of points where the rank of the differential of $F$ is smaller then $k$. Let $X_\textrm{reg}$ the set of regular points in $X$ and set 
$$
U := X_{\textrm{reg}} \setminus F(R).
$$ 
Then, $U$ and $F^{-1}(U)$ are nonempty Zariski open sets such that  
$$
F|_{F^{-1}(U)}\colon F^{-1}(U) \to U
$$ 
is a finite \'etale cover.  

\medskip

To  prove (i), we observe that if $(M,\omega)$ is K\"ahler hyperbolic, then the form $F^*(\omega)$ is semi-positive and it is positive on $F^{-1}(U)$, therefore by Lemma \ref{bounded} $(N,F^*(\omega))$ is semi-K\"ahler hyperbolic.

\smallskip

To prove (ii), we note that if $k=m$ then $X=M$. If $M$ is weakly K\"ahler hyperbolic and $\mu \in \mathcal{W}_M$ then $F^*(\mu) \in \mathcal{W}_N$, which is thus nonempty. 

\smallskip

We now prove (iii). The fact that $X=F(N)$ is not contained in $Z_M$ gives us a $\mu \in \mathcal{W}_M$ such that $X\not\subset E_{nK}([\mu])$. We have that $F^*[\mu]$ is a nef class. Since $F$ has generically finite fibers, the integral $\int_N F^*(\mu)^k$ is positive if and only if the integral $\int_X \mu^k$ is. The latter being the intersection number $[\mu]^k\cdot X$, it must be nonzero by \cite[Theorem 1.1]{CT15}, for $X$ is not contained in $E_{nK}([\mu])$. Thus $F^*[\mu]$ is also big and $F^*[\mu]\in\mathcal W_N$ by Lemma \ref{bounded}.
\end{proof}

\begin{quest}
Can we replace the condition that $F$ is a generically finite holomorphic map with generically finite \emph{meromorphic} map? This would amount to understand if and how Kähler hyperbolic type properties descend through a proper modification. 
\end{quest}

We plan to address this question which is essentially about whether the property of being weakly Kähler hyperbolic is of birational nature in a forthcoming note.

So, if $M$ is weakly Kähler hyperbolic, since $Z_M$ is always a proper subvariety, submanifolds passing through a general point of $M$ have to be weakly Kähler hyperbolic. Especially, any rational or elliptic curve in $M$ is contained in $Z_M$. We shall see later that a weakly K\"ahler hyperbolic manifold $M$ is of general type, and therefore in this case if $Z_M$ does not contain rational curves, then $K_M$ is ample.

\medskip

Let us finish this subsection with the following proposition, whose goal is to construct a bounded primitive (on the universal cover) of closed real $(1,1)$-forms which represents the cohomology class of a real Cartier divisor, after a finite ramified covering.

\begin{prop}\label{prop:modiftoboundedprimitive}
Let $X$ be a complex projective manifold. There exists a smooth projective manifold $Z$ together with a finite surjective morphism $f\colon Z \to X$ such that, 
 for every $\mathbb R$-divisor $D$ and every closed real $(1,1)$-form $u$ on $X$ representing the cohomology class $[D]\in H^{1,1}(X, \mathbb R)$, we have that  $\pi^*f^*u$ has a bounded primitive, where $\pi\colon\tilde{Z}\to Z$ denotes the universal covering space of $Z$. 
\end{prop}

\begin{proof}
Given any ample line bundle $A\to X$, fix a projective embedding $\iota\colon X\to\mathbb P^N$ so that $\iota^*\mathcal O(1)\simeq A^{\otimes m}$ for some $m\in\mathbb N$. Let $\Xi_N$ be a compact complex manifold uniformized by the unit ball $\mathbb B^N\subset\mathbb C^N$. Observe that the latter always exists by \cite[Corollary of Theorem A]{Bor63} and is projective, since it has ample cotangent bundle. 

By Noether's lemma, there exists a finite surjective morphism $\nu\colon\Xi_N \to \mathbb{P}^N$. By a theorem of Kleiman \cite{Kle74} (see also \cite[Theorem III.10.8]{Har77}) acting with a generic projective automorphism, we can assume that the subscheme $\nu^{-1}(\iota(X))\subset \Xi_N$ which is isomorphic to $X\times_{\mathbb{P}^N} \Xi_N$ is in fact a smooth submanifold. Call $Z_A$ a nonempty connected component of such fiber product.

Then, the natural map $f_A: Z_A \to X$ is a finite surjective map and $\xi:=\nu|_{Z_A}= \iota \circ f_A$. In particular $f^*_A A^{\otimes m}=\xi^*\mathcal O(1)$.  

Since $\Xi_N$ has negative Riemannian sectional curvature, every closed form $v$ on $\Xi_N$ is such that $p^*v$ has a bounded primitive, where $p\colon\mathbb B^N \to \Xi_N$ denotes the universal covering space. 

For every de Rham representative $u$ of $c_1(A)$ we have at the level of cohomology $[f^*_Au]=[\xi^*w]=[\nu^* w]|_{Z_A}$ for some $w\in\frac 1m\,c_1\bigl(\mathcal O(1)\bigr)$, and using Lemma \ref{bounded} and Remark \ref{rem:khypcohomological}, it follows that $q^*f^*_Au $ has a bounded primitive, where $q\colon\tilde Z_A \to Z_A$ denotes the universal covering space, since $p^*\nu^* w$ has.

To finish the proof, let $A_1,\dots, A_d\to X$ be a finite set of ample line bundles such that their real Chern classes form a basis of the real Néron--Severi space $NS(X)_{\R} \subset H^{1,1}(X,\R)$, which is spanned by the numerical classes of $\R$-divisors. This exists because the ample cone is open and thus contains a basis consisting of classes of ample divisors. 

We now iterate the above construction, starting with $A_1\to X$. We take $f_{A_1}^*A_2\to Z_{A_1}$, and observe that $f_{A_1}^*A_2$ is still ample since $f_{A_1}^*$ is a finite morphism. We thus get a smooth projective manifold $Z_{A_1,A_2}$ together with a finite surjective morphism $f_{A_1,A_2}\colon Z_{A_1,A_2}\to X$ such that every de Rham representative of $c_1(A_1)$ or $c_1(A_2)$ has a bounded primitive once pulled-back on the universal cover of $Z_{A_1,A_2}$. 

Continuing in this fashion, we come up with a smooth projective manifold $Z:=Z_{A_1,\dots,A_d}$, and with a finite surjective map $f:=f_{A_1,\dots,A_d}\colon Z\to X$ which has the required properties for all classes $\{c_1(A_i)\}_{1\le i \le d}$, and hence for all classes $[D] \in NS(X)_{\R}$. 
\end{proof}

\begin{cor} There exists a complex projective manifold $Z$ together with a nef and big divisor $D$ 
such that every closed positive current representing 
the cohomology class $[D]\in H^{1,1}(X, \R)$ has no representative with locally bounded plurisubharmonic local potentials and for every smooth closed $(1,1)$-form $u$ on $X$ representing $[D]$,  $\pi^* u$ 
has a bounded primitive, where $\pi\colon\widetilde{Z}\to Z$ denotes the universal covering space of $Z$. 
\end{cor}

\begin{proof}
Let $X$ be a projective manifold on which there is a nef and big divisor $D'$ 
such that every closed positive current representing 
$[D']$ has locally unbounded local plurisubharmonic potentials. For the construction of such a manifold with a divisor with these properties, see \textsl{e.g.} \cite[Example 5.4]{BEGZ10}. 

Now, apply Proposition \ref{prop:modiftoboundedprimitive} to the pair $(X,D')$. The divisor $D=\mu^* D'$ has the required property for, if $[D]$ had a representative with locally bounded plurisubharmonic local potentials, then
$[D']=\frac{1}{\deg(\mu)} \mu_*[D]$ would, too. 
\end{proof}

\begin{rem}
Unfortunately, the manifold $Z$ we construct is K\"ahler hyperbolic. Hence, we have to leave open the question whether there exists a weakly K\"ahler hyperbolic manifold which is not semi-K\"ahler hyperbolic. 
\end{rem}

\section{Spectral Gap and Nonvanishing for Degenerate K\"ahler Hyperbolic Metrics}

In this section we show how the notion of weakly K\"ahler hyperbolicity we introduced earlier entails various consequences for the spectrum of the Hodge--Kodaira Laplacian acting on the space of $L^2$ $(p,0)$-forms of the universal covering of a suitable modification $\nu\colon M'\to M$ of our starting manifold $M$. We will prove that it yields the absence of zero in the spectrum of $\Delta_{\overline{\partial},p,0}$ acting on $L^2\Omega^{p,0}(\tilde{M'})$, and consequently the vanishing of $L^2$ holomorphic $p$-forms on the universal covering $\tilde{M'}$ with $0\leq p<m$. 

In the second part, joining these vanishing results with the argument given by Gromov in \cite{Gro91} and the birational invariance of $L^2$-Hodge $(p,0)$-numbers, we show that the space of $L^2$ holomorphic $m$ forms on the universal cover $\tilde{M}$ of $M$ is infinite dimensional provided $M$ is weakly K\"ahler hyperbolic.

\subsection{Spectral gap}

We first deal with the announced spectral gap.

\begin{thm}
\label{thm:spectralgap}
Let $M$ be a compact K\"ahler manifold of complex dimension $\dim M=m$, and suppose that $M$ is weakly Kähler hyperbolic. 

Then, there exists a compact Kähler manifold $(M',\omega)$ together with a modification $\nu\colon M'\to M$ such that the following holds.

Let $\pi\colon(\tilde{M'},\tilde{\omega})\to (M',\omega)$ be the universal covering of $M'$ endowed with the pull-back metric $\tilde{\omega}:=\pi^*\omega$. For $0\leq p\le m-1$, let 
$$
\Delta_{\overline{\partial},p,0}\colon L^2\Omega^{p,0}(\tilde{M'},\tilde{\omega})\rightarrow L^2\Omega^{p,0}(\tilde{M'},\tilde{\omega})
$$ 
be the closure of $\Delta_{\overline{\partial},p,0}\colon\Omega_c^{p,0}(\tilde{M'})\rightarrow \Omega_c^{p,0}(\tilde{M'})$.

Then, zero is not in the spectrum $\sigma(\Delta_{\overline{\partial},p,0})$ for any $0\leq p\le m-1$, and furthermore $\Delta_{\overline{\partial},m,0}$ has closed range.
\end{thm}

Observe that, by Example \ref{ex:dimone}, the above theorem in dimension one follows immediately from the work of Gromov, so that in its proof we can restrict our attention to dimension at least two (this will be used in the proof of Lemma \ref{lem:maintool}). We need now some preliminary results. 

\begin{lem}\label{lem:forms} 
Let $(M,\omega)$ be a Hermitian manifold of dimension $m$, let $0 \leq k \leq m$ be an integer, and let respectively $\alpha$ be a $(2m - 2k)$-form and $\beta$ be a $(m-k,0)$-form. Then, we have
$$
|\alpha\wedge\omega^k|_\omega\le\frac{m!}{(m-k)!}\,|\alpha|_\omega,\quad |\beta|_\omega^2\,\frac{\omega^m}{m!}=i^{(m-k)^2}\beta\wedge\bar\beta\wedge\frac{\omega^k}{k!}.
$$
\end{lem}

\begin{proof}
Call $\alpha'$ the bidegree $(m-k,m-k)$ component of $\alpha$. Observe first that
$$
\alpha\wedge\omega^k=\alpha'\wedge\omega^k=\frac{k!}{(m-k)!}\,\alpha'\wedge *\omega^{m-k}=\frac{k!}{m!(m-k)!}\,\langle\alpha',\omega^{m-k}\rangle_\omega\,\omega^m.
$$
By Cauchy--Schwarz we have
$$
|\alpha\wedge\omega^k|_\omega=\frac{k!}{(m-k)!}\,|\langle\alpha',\omega^{m-k}\rangle_\omega|\le k!\,|\alpha'|_\omega\left|\frac{\omega^{m-k}}{(m-k)!}\right|_\omega\le\frac{m!}{(m-k)!}\,|\alpha|_\omega,
$$
since $|\omega^k|_\omega=k!\binom{m}{k}$, and obviously $|\alpha'|_\omega\le|\alpha|_\omega$.

The equality for the squared norm of $\beta$ times the volume form is just a straightforward pointwise computation in (unitary) coordinates.
\end{proof}

\begin{prop}\label{prop:normcontrol}
On a compact Kähler manifold $(M,\omega)$ with Kähler universal cover $(\tilde M,\tilde\omega)\to (M,\omega)$, given any $p,q\in\{0,...,m\}$, let 
$$
\Delta_{\overline{\partial},p,q}\colon L^2\Omega^{p,q}(\tilde{M},\tilde{\omega})\rightarrow L^2\Omega^{p,q}(\tilde{M},\tilde{\omega})
$$ 
be the closure of  $\Delta_{\overline{\partial},p,q}\colon\Omega_c^{p,q}(\tilde{M})\rightarrow \Omega_c^{p,q}(\tilde{M})$ and let $\{E(\lambda)\}_{\lambda}$ be its spectral resolution. For any fixed $\lambda_0 > 0$, there exists $\varepsilon_0(\lambda_0)>0$ such that if $0 < \varepsilon < \varepsilon_0(\lambda_0),$   and $U_{\varepsilon}$ is an  open set   in $M$ with 
$$
\operatorname{vol}_\omega(U_\varepsilon)<\varepsilon,
$$
we have:
$$
\int_{\tilde{U_{\varepsilon}}}|\eta|^2_{\tilde{\omega}}\dvol_{\tilde{\omega}}\leq \int_{{\tilde{M}\setminus \tilde{U_{\varepsilon}}}}|\eta|^2_{\tilde{\omega}}\dvol_{\tilde{\omega}}
$$
for all $\eta\in \im(E(\lambda_0))$, where $\tilde{U_{\varepsilon}}$ is the preimage of $U_{\varepsilon}$ through $\pi\colon\tilde{M}\rightarrow M$.
\end{prop}

The above proposition will follow easily from the next. 

\begin{lem}
\label{lem:maintool}
Let $D_1,...,D_N$ and $B_1,...,B_N$ be  finite sequences of open subsets of $\tilde{M}$ such that
\begin{enumerate}
\item $D_i$ is a fundamental domain of $\pi\colon\tilde{M}\rightarrow M$ for each $i=1,...,N$;
\item $B_i$ is a relatively compact open subset with smooth boundary of $\tilde{M}$ such that $\overline{B_i}\subset D_i$, for each $i=1,...,N$;
\item $\{\pi(B_i)\}_{1,...,N}$ is an open cover of $M$.
\end{enumerate}
Given $\lambda_0 > 0$, there exists $\varepsilon_0(\lambda_0)>0$ such that if $0 < \varepsilon < \varepsilon_0(\lambda_0)$ and $U_{\varepsilon}$ is an  open set in $M$ satisfying  
$$
\operatorname{vol}_\omega(U_\varepsilon)<\varepsilon,
$$
then we have
$$
\int_{\pi_1(M).B_i\cap \tilde{U_{\varepsilon}}}|\eta|_{\tilde{\omega}}^2\dvol_{\tilde\omega}\leq \frac{1}{2N}\int_{\tilde{M}}|\eta|^2_{\tilde{\omega}}\dvol_{\tilde{\omega}},
$$ 
for any $\eta\in \im(E(\lambda_0))$ and $i=1,...,N$. 
\end{lem}

For the standard construction of such fundamental domains see for instance \cite[\S 3.6.1]{MM07}.

\begin{proof}
Throughout the proof let us fix arbitrarily a pair $B_i\subset D_i$ that for simplicity we will denote with $B\subset D$. Let $\eta\in \im(E(\lambda_0))$ be fixed. Since $\eta\in \im(E(\lambda_0))$ we have that  
$$
\int_{\tilde{M}}|\Delta_{\overline{\partial},p,q}\eta|^2_{\tilde{\omega}}\dvol_{\tilde{\omega}}\leq \lambda_0^{2}\int_{\tilde{M}}|\eta|^2_{\tilde{\omega}}\dvol_{\tilde{\omega}}.
$$
Therefore, 
$$\sum_{\gamma\in \pi_1(M)}\int_{\gamma D}(|\Delta_{\overline{\partial},p,q}\eta|^2_{\tilde{\omega}}-\lambda_0^{2}|\eta|^2_{\tilde{\omega}})\dvol_{\tilde{\omega}}\leq 0
$$ 
and thus there exists at least an element $\overline{\gamma}\in \pi_1(M)$ such that 
$$
\int_{\overline{\gamma} D}(|\Delta_{\overline{\partial},p,q}\eta|^2_{\tilde{\omega}}-\lambda_0^{2}|\eta|^2_{\tilde{\omega}})\dvol_{\tilde{\omega}}\leq 0.
$$ 
Let us define $S(\eta)\subset \pi_1(M)$ as 
$$
S(\eta):=\left\{\gamma \in \pi_1(M): \int_{\gamma D}|\Delta_{\overline{\partial},p,q}\eta|^2_{\tilde{\omega}}\dvol_{\tilde{\omega}}\leq 4N\lambda_0^{2}\int_{\gamma D}|\eta|_{\tilde{\omega}}^2\dvol_{\tilde{\omega}}\right\}.
$$ 
We have 
\begin{multline*}
\int_{\tilde{M}}|\eta|^2_{\tilde{\omega}}\dvol_{\tilde{\omega}}
\geq \lambda_0^{-2}\int_{\tilde{M}}|\Delta_{\overline{\partial},p,q}\eta|^2_{\tilde{\omega}}\dvol_{\tilde{\omega}} \\
\geq \sum_{\gamma \notin S(\eta)} \lambda_0^{-2}\int_{\gamma D}|\Delta_{\overline{\partial},p,q}\eta|^2_{\tilde{\omega}}\dvol_{\tilde{\omega}}
\geq  \sum_{\gamma \notin S(\eta)} 4N\int_{\gamma D}|\eta|^2_{\tilde{\omega}}\dvol_{\tilde{\omega}}.
\end{multline*}
Thus, we can deduce that, for any $\varepsilon>0$,
\begin{equation}
\label{noS}
\sum_{\gamma \notin S(\eta)} \int_{\gamma D\cap \tilde{U_{\varepsilon}}}|\eta|^2_{\tilde{\omega}}\dvol_{\tilde{\omega}}\leq \frac{1}{4N}\int_{\tilde{M}}|\eta|^2_{\tilde{\omega}}\dvol_{\tilde{\omega}}.
\end{equation}
Let us consider now any $\gamma\in S(\eta)$ and let us label with $\eta_{\gamma}$ the pullback of $\eta$ through $\gamma^{-1}$, that is $\eta_{\gamma}:=(\gamma^{-1})^*\eta$.
Since $\gamma^{-1}\colon\tilde{M}\rightarrow \tilde{M}$ is both a biholomorphism and an isometry we have 
\begin{multline}
\label{control}
\int_D|\Delta_{\overline{\partial},p,q}\eta_{\gamma}|_{\tilde{\omega}}^2\dvol_{\tilde{\omega}}
=\int_{\gamma D}|\Delta_{\overline{\partial},p,q}\eta|_{\tilde{\omega}}^2\dvol_{\tilde{\omega}} \\
\leq 4N\lambda_0^{2}\int_{\gamma D}|\eta|_{\tilde{\omega}}^2\dvol_{\tilde{\omega}}
=4N\lambda_0^{2}\int_D|\eta_{\gamma}|_{\tilde{\omega}}^2\dvol_{\tilde{\omega}}.
\end{multline}
Thanks to the elliptic estimates, see \textsl{e.g.} \cite[Lemma 1.1.17]{Les97}, we know that there exists a positive constant $C$ such that for any  $\psi\in \mathcal{D}(\Delta_{\overline{\partial},p,q})$ we have $\psi|_B\in W^{2,2}\Omega^{p,q}(B,\tilde{\omega})$ and 
$$
\|\psi\|_{W^{2,2}\Omega^{p,q}(B,\tilde{\omega})}\leq C(\|\psi\|_{L^2\Omega^{p,q}(D,\tilde{\omega})}+\|\Delta_{\overline{\partial},p,q}\psi\|_{L^2\Omega^{p,q}(D,\tilde{\omega})}).
$$ 
Thanks to the Sobolev embedding theorem (see for instance \cite{Aub98}), since $m\ge 2$, we have a continuous inclusion 
$$
W^{1,2}\Omega^{p,q}(B,\tilde{\omega})\hookrightarrow L^{\frac{2m}{m-1}}\Omega^{p,q}(B,\tilde{\omega}),
$$ 
and hence we can deduce that there exists a possibly new positive constant ---which we still call $C$--- such that for any $\psi\in \mathcal{D}(\Delta_{\overline{\partial},p,q})$ it holds
\begin{equation}
\label{Sob}
\|\psi\|_{L^{\frac{2m}{m-1}}\Omega^{p,q}(B,\tilde{\omega})}\leq C(\|\psi\|_{L^2\Omega^{p,q}(D,\tilde{\omega})}+\|\Delta_{\overline{\partial},p,q}\psi\|_{L^2\Omega^{p,q}(D,\tilde{\omega})}).
\end{equation}
Thus, for $\gamma\in S(\eta)$ we get, thanks to \eqref{control} and \eqref{Sob},
\begin{multline*}
\|\eta_{\gamma}\|_{L^{\frac{2m}{m-1}}\Omega^{p,0}(B,\tilde{\omega})}
\leq C(\|\eta_{\gamma}\|_{L^2\Omega^{p,q}(D,\tilde{\omega})}+\|\Delta_{\overline{\partial},p,q}\eta_{\gamma}\|_{L^2\Omega^{p,q}(D,\tilde{\omega})}) \\ 
\leq  C(1+2\sqrt{N}\lambda_0)\|\eta_{\gamma}\|_{L^2\Omega^{p,q}(D,\tilde{\omega})},
\end{multline*} 
and therefore 
\begin{equation}
\label{Sobolev}
\begin{aligned}
\int_{B\cap \tilde{U_{\varepsilon}}}|\eta_{\gamma}|^2_{\tilde{\omega}}\dvol_{\tilde{\omega}} & \leq \left(\int_{B\cap \tilde{U_{\varepsilon}}}\dvol_{\tilde{\omega}}\right)^{\frac{1}{m}}\left(\int_{B\cap \tilde{U_{\varepsilon}}}(|\eta_{\gamma}|^2_{\tilde{\omega}})^{\frac{m}{m-1}}\dvol_{\tilde{\omega}}\right)^{\frac{m-1}{m}} \\
& = \bigl(\vol_{\tilde{\omega}}(B\cap \tilde{U_{\varepsilon}})\bigr)^{\frac{1}{m}}\left(\int_{B\cap \tilde{U_{\varepsilon}}}|\eta_{\gamma}|_{\tilde{\omega}}^{\frac{2m}{m-1}}\dvol_{\tilde{\omega}}\right)^{\frac{m-1}{m}} \\
& =\bigl(\vol_{\tilde{\omega}}(B\cap \tilde{U_{\varepsilon}})\bigr)^{\frac{1}{m}}\|\eta_{\gamma}\|^2_{L^{\frac{2m}{m-1}}\Omega^{p,q}(B\cap \tilde{U_{\varepsilon}},\tilde{\omega})} \\
& \leq \bigl(\vol_{\tilde{\omega}}(B\cap \tilde{U_{\varepsilon}})\bigr)^{\frac{1}{m}}C^2(1+2\sqrt{N}\lambda_0)^2\|\eta_{\gamma}\|_{L^2\Omega^{p,q}(D,\tilde{\omega})}^2 \\
& =\bigl(\vol_{\tilde{\omega}}(B\cap \tilde{U_{\varepsilon}})\bigr)^{\frac{1}{m}}C^2(1+2\sqrt{N}\lambda_0)^2\int_{D}|\eta_{\gamma}|^2_{\tilde{\omega}}\dvol_{\tilde{\omega}}.
\end{aligned}
\end{equation}
Finally let us pick $\varepsilon$ in such a way that 
$$
\bigl(\vol_{\tilde{\omega}}(B\cap \tilde{U_{\varepsilon}})\bigr)^{\frac{1}{m}}C^2(1+2\sqrt{N}\lambda_0)^2<\frac{1}{4N},
$$ 
in order to obtain
$$
\begin{aligned}
\int_{\pi_1(M).B\cap \tilde{U_{\varepsilon}}}|\eta|^2_{\tilde{\omega}}\dvol_{\tilde{\omega}} & =\sum_{\gamma \in S(\eta)}\int_{\gamma B\cap \tilde{U_{\varepsilon}}}|\eta|^2_{\tilde{\omega}}\dvol_{\tilde{\omega}}+\sum_{\gamma \notin S(\eta)}\int_{\gamma B\cap \tilde{U_{\varepsilon}}}|\eta|^2_{\tilde{\omega}}\dvol_{\tilde{\omega}} \\
\textrm{(by \eqref{noS})} & \leq \sum_{\gamma \in S(\eta)}\int_{\gamma B\cap \tilde{U_{\varepsilon}}}|\eta|^2_{\tilde{\omega}}\dvol_{\tilde{\omega}}+\frac{1}{4N}\int_{\tilde{M}}|\eta|^2_{\tilde{\omega}}\dvol_{\tilde{\omega}} \\
& = \sum_{\gamma \in S(\eta)}\int_{\gamma (B\cap \tilde{U_{\varepsilon}})}|\eta|^2_{\tilde{\omega}}\dvol_{\tilde{\omega}}+\frac{1}{4N}\int_{\tilde{M}}|\eta|^2_{\tilde{\omega}}\dvol_{\tilde{\omega}} \\
& = \sum_{\gamma \in S(\eta)}\int_{B\cap \tilde{U_{\varepsilon}}}|\eta_{\gamma}|^2_{\tilde{\omega}}\dvol_{\tilde{\omega}}+\frac{1}{4N}\int_{\tilde{M}}|\eta|^2_{\tilde{\omega}}\dvol_{\tilde{\omega}} \\ 
\end{aligned}
$$
$$
\begin{aligned}
\quad \textrm{(by \eqref{Sobolev})} & \leq \sum_{\gamma \in S(\eta)}\frac{1}{4N}\int_{D}|\eta_{\gamma}|^2_{\tilde{\omega}}\dvol_{\tilde{\omega}}+\frac{1}{4N}\int_{\tilde{M}}|\eta|^2_{\tilde{\omega}}\dvol_{\tilde{\omega}} \\
& = \sum_{\gamma \in S(\eta)}\frac{1}{4N}\int_{\gamma D}|\eta|^2_{\tilde{\omega}}\dvol_{\tilde{\omega}}+\frac{1}{4N}\int_{\tilde{M}}|\eta|^2_{\tilde{\omega}}\dvol_{\tilde{\omega}} \\
& \leq \frac{1}{4N}\int_{\tilde{M}}|\eta|^2_{\tilde{\omega}}\dvol_{\tilde{\omega}}+\frac{1}{4N}\int_{\tilde{M}}|\eta|^2_{\tilde{\omega}}\dvol_{\tilde{\omega}} \\
& =\frac{1}{2N}\int_{\tilde{M}}|\eta|^2_{\tilde{\omega}}\dvol_{\tilde{\omega}}.
\end{aligned}
$$
Repeating the above procedure for each pair $B_i\subset D_i$ and choosing $\varepsilon>0$ in such a way that 
$$
\bigl(\vol_{\tilde{\omega}}(B_i\cap \tilde{U_{\varepsilon}})\bigr)^{\frac{1}{m}}C^2(1+2\sqrt{N}\lambda_0)^2<\frac{1}{4N}, \quad\textrm{for $i=1,\dots,N$}, 
$$
we can conclude that for any arbitrarily fixed $\lambda_0>0$ there exists $\varepsilon_0(\lambda_0)>0$ such that if $0 < \varepsilon < \varepsilon_0(\lambda_0)$ and $U_{\varepsilon}$ is an  open set in $M$ satisfying  
$$
\operatorname{vol}_\omega(U_\varepsilon)<\varepsilon,
$$
then we have
$$
\int_{\pi_1(M).B_i\cap \tilde{U_{\varepsilon}}}|\eta|_{\tilde{\omega}}^2\dvol_{\omega}\leq \frac{1}{2N}\int_{\tilde{M}}|\eta|^2_{\tilde{\omega}}\dvol_{\tilde{\omega}},
$$ 
for any $\eta\in \im(E(\lambda_0))$ and $i=1,...,N$, as desired.
\end{proof}

With Lemma \ref{lem:maintool} at our disposal, we can now prove Proposition \ref{prop:normcontrol}.

\begin{proof}[Proof of Proposition \ref{prop:normcontrol}]
We have
\begin{multline*}
\int_{\tilde{U_{\varepsilon}}}|\eta|^2_{\tilde{\omega}}\dvol_{\tilde{\omega}}\leq \sum_{i=1}^N\int_{\pi_1(M).B_i\cap \tilde{U_{\varepsilon}}}|\eta|^2_{\tilde{\omega}}\dvol_{\tilde{\omega}} \\ 
\textrm{(by Lemma \eqref{lem:maintool})}\ \leq  \sum_{i=1}^N\frac{1}{2N}\int_{\tilde{M}}|\eta|^2_{\tilde{\omega}}\dvol_{\tilde{\omega}}= \frac{1}{2}\int_{\tilde{M}}|\eta|^2_{\tilde{\omega}}\dvol_{\tilde{\omega}}.
\end{multline*}
Thus, 
$$
\int_{\tilde{U_{\varepsilon}}}|\eta|^2_{\tilde{\omega}}\dvol_{\tilde{\omega}}\leq \frac{1}{2}\int_{\tilde{M}\setminus \tilde{U_{\varepsilon}}}|\eta|^2_{\tilde{\omega}}\dvol_{\tilde{\omega}}+\frac{1}{2}\int_{\tilde{U_{\varepsilon}}}|\eta|^2_{\tilde{\omega}}\dvol_{\tilde{\omega}}
$$ 
and so we reach the desired conclusion 
$$
\int_{\tilde{U_{\varepsilon}}}|\eta|^2_{\tilde{\omega}}\dvol_{\tilde{\omega}}\leq \int_{\tilde{M}\setminus \tilde{U_{\varepsilon}}}|\eta|^2_{\tilde{\omega}}\dvol_{\tilde{\omega}}.
$$
\end{proof}




We are finally in position to prove Theorem \ref{thm:spectralgap}.

\begin{proof}[Proof of Theorem \ref{thm:spectralgap}]

The first step is to produce the required modification $\nu\colon M'\to M$, for which we proceed as follows. We pick $[\mu]\in\mathcal W_M$ and select a Kähler current $T\in[\mu]$, which exists since $[\mu]$ is big. Now, by the proof of \cite[Proposition 2.3]{Bou04}, or the proof of \cite[Theorem 3.4 and Lemma 3.5]{DP04} there exists a modification $\nu\colon M'\to M$, an effective $\mathbb R$-divisor $E$, and a Kähler form $\omega$ on $M'$ such that the pull-back $\nu^*T$ is cohomologous to $\omega+[E]$, where $[E]$ denotes the current of integration along $E$. Moreover, one has that $E=\nu^{-1}\bigl(\nu(E)\bigr)$, and the restriction $\nu|_{M'\setminus E}$ is a biholomorphism onto $M\setminus\nu(E)$.

Set $\mu' = \nu^*\mu$. Since $[\mu]$ is nef and big, so is $[\mu']$ and $[\mu']\in\mathcal W_{M'}$. From now on we shall be working on $M'$, so for the sake of simplicity let us replace $M$ with $M'$ and $[\mu]$ with $[\mu']$.

Define, for $t\ge 0$, the function
$$
\lambda(t) := \log \frac{\int_{M} (\mu + t\,\omega)^m}{ \int_{M}  {\omega}^m}.
$$
For $0 \leq  t \leq 1$, $\lambda(t)$ is bounded independently of $t$. For any fixed $0\le t\le 1$, we now consider the Monge--Ampère equation 
$$
\bigl(\mu + t\,\omega  + i \partial \bar \partial \phi_t\bigr)^m = e^{\lambda(t)} \,{\omega}^m,
$$ 
where $(\bullet)^m$ denotes the non-pluripolar $m$-th power. Since $[\mu]$ is nef, 
$$
\int_M e^{\lambda(t)} {\omega}^m =  \int_M (\mu + t\,\omega)^m
$$ 
is the volume of the class $[\mu + t\,\omega]$.

By \cite[Theorem 3.1]{BEGZ10}, the above equation has a unique $\mu$-psh solution $\phi_t$  such that 
 $ \max_{M} \phi_t = 0$.  Moreover by \cite{Yau78}, for $t > 0$, $\phi_t$ is smooth on $M$ and $\mu+ t\,\omega +  i\partial \bar \partial \phi_t$ is a K\"ahler form on $M$.
   
Next, given a smooth real closed $(1,1)$-form $\eta$ on $M$, set
$$
V_\eta:=\sup\{\tau\mid\textrm{$\tau$ is $\eta$-psh and $\tau\le 0$}\}.
$$ 
By definition, $\tau$ is $\eta$-psh if $\tau$ is an upper semi-continuous function $\tau\colon M\to [-\infty,+\infty)$ such that $\eta+i\partial\bar\partial\tau$ is a positive $(1,1)$-current. Observe that $V_\eta$ is a $\eta$-psh function and that  $V_{\eta+t\omega}\downarrow V_\eta$ as the positive real parameter $t$ decreases to $0$.

Now, (the proof of) \cite[Theorem 4.1]{BEGZ10} shows that there exists a constant $L>0$ such that  $- L + V_{\mu + t\, \omega} \leq \phi_t \leq 0$, for $0 < t \leq 1$. Moreover, $\phi_t$ is uniformly bounded independently of $t$ on compact sets in $M\setminus E$ since the current $T$ we started with was smooth on the complement of $\nu(E)$.

We can then repeat the proof of \cite[Theorem 5.1]{BEGZ10} to show that for $0 < t < 1$ sufficiently small, the $\omega$-Laplacian of $\phi_t$ is bounded independently of $t$ on compacta of $M\setminus E$. It follows from standard theory of complex Monge--Ampère equations (see \textsl{e.g.} \cite{Siu87}) that, for $0 < \beta < 1$ and $k$ any non negative integer, the $C^{k,\beta}(W)$-norm of $\phi_t$ is bounded independently of $t$ for every open set $W$ relatively compact in $M\setminus E$.  

Now, by Hartogs theorem (see \cite{GZ07}), for a given sequence $\{t_n\}\subset (0,1)$ converging to $0$, up to subsequences, we have that $\phi_{t_n}$ converges in $L^1$ to some $\mu$-plurisubharmonic function $\psi$ with $\max_{M}\psi = 0$. 
On the other hand for each $\{t_n\}\subset (0,1)$ converging to $0$, up to subsequences, $\phi_{t_{n}}$ converges in $C^{k,\beta}(W)$. In particular, the function $\psi$ above is smooth on $M\setminus E$ and solves the corresponding Monge--Ampère equation there. Since $E$ is a pluripolar set, we conclude that $\psi = \phi_0$ and that $\phi_t$ converges to $\phi_0$ in each $C^{k,\beta}(W)$ as above. Observe that, in particular, $\mu + i \partial \bar{\partial} \phi_0$ is a K\"ahler form on $M\setminus E$. 

Next, we are going to show that $0$ is not in the spectrum of the $\tilde\omega$-Laplacian $\Delta_{\bar\partial,p,0}$ on the Kähler universal cover $(\tilde M,\tilde\omega)$ of $(M,\omega)$ for all $0 \leq p \leq m-1$. 

We consider the smooth forms $ {\mu}_t := \mu+i\partial \bar \partial \phi_t$ so that ${\mu}_t + t\,\omega$ is a K\"ahler form on $M$ for all $t > 0$. 
Let $\alpha$ be a smooth bounded primitive of the pull back of $\mu$ to the universal covering $\tilde M$ (which exists by Lemma \ref{bounded} applied to the morphism $\nu\colon M'\to M$), so that $ \alpha_t := \alpha + i  \bar \partial \tilde\phi_t$ is a smooth bounded primitive of $\tilde{\mu_t}=\pi^*\mu_t$, where $\tilde\phi_t=\phi_t\circ\pi$.

Now, let $\eta$ be a smooth $(p,0)$-form with compact support on $(\tilde M,\tilde\omega)$, and let $r$ be an integer with $0 \leq r \leq m-p-1$. We have the following equality of top degree forms:
\begin{multline*}
d(\eta\wedge\overline{\eta}\wedge \alpha_t\wedge \tilde{\mu_t}^{m-p-r-1} \wedge \tilde\omega^r)
 \\ =d(\eta\wedge\overline{\eta})\wedge \alpha_t\wedge \tilde{\mu_t}^{m-p- r - 1} \wedge \tilde\omega^r +\eta\wedge\overline{\eta}\wedge \tilde{\mu_t}^{m-p-r} \wedge \tilde\omega^r  \\
= d\eta\wedge\overline{\eta}\wedge \alpha_t\wedge \tilde{\mu_t}^{m-p-1} \wedge \tilde\omega^r +(-1)^p\eta\wedge d\overline{\eta}\wedge \alpha_t\wedge \tilde{\mu_t}^{m-p- r -1} \wedge \tilde\omega^r \\
+\eta\wedge\overline{\eta}\wedge \tilde{\mu_t}^{m-p-r} \wedge \tilde\omega^r.
\end{multline*}

By Stokes, and using Lemma \ref{lem:forms}, we find: 
\begin{equation}\label{ineq-quasi}
\begin{aligned}
\biggl|\int_{\tilde M} i^{p^2} \eta\wedge\bar\eta & \wedge \tilde \mu_t^{m-p-r} \wedge \tilde\omega^r  \biggr| = \biggl| \int_{\tilde M}i^{p^2} d \eta\wedge\overline{\eta}\wedge \alpha_t\wedge \tilde{\mu_t}^{m-p-r-1}\wedge\tilde{\omega}^r \\ &\qquad\qquad\qquad\qquad\quad+i^{p^2}(-1)^p\eta\wedge d\overline{\eta}\wedge \alpha_t\wedge \tilde{\mu_t}^{m-p-r-1} \wedge \tilde{\omega}^r  \biggr|  \\
& \leq\frac{m!}{(m-r)!}\biggl(  \int_{\tilde{M}}\bigl|d\eta\wedge\overline{\eta}\wedge \alpha_t\wedge \tilde{\mu_t}^{m-p-r-1}\bigr|_{\tilde\omega}\dvol_{\tilde\omega} \\
&\qquad+  \int_{\tilde{M}}\bigl|\eta\wedge d\overline{\eta}\wedge \alpha_t\wedge \tilde{\mu_t}^{m-p-r-1}\bigr|_{\tilde\omega}
\dvol_{\tilde{\omega}} \biggr)\\
& \leq 2 \frac{m!}{(m-r)!} \int_{\tilde{M}}|d\eta|_{\tilde\omega}\,|\overline{\eta}|_{\tilde\omega}\, |\alpha_t|_{\tilde\omega}\,|\tilde{\mu_t}^{m-p-r-1}|_{\tilde\omega}\,\dvol_{\tilde\omega}.
\end{aligned}
\end{equation} 

Summarizing, there exists a positive function $C_r(t)$ defined by
$$
C_r(t) :=2 \sqrt{2}\,\frac{m!}{(m-r)!}  \|\alpha_t\|_{L^{\infty}\Omega^1(\tilde{M},\tilde\omega)}\|\tilde{\mu_t}^{m-p-r-1}\|_{L^{\infty}\Omega^{2(m-p-1)}(\tilde{M},\tilde\omega)}
$$   
such that for each $\eta\in \Omega_c^{p,0}(\tilde M)$ we have 
\begin{multline*}
\left| \int_{\tilde M} i^{p^2} \eta\wedge\overline{\eta}\wedge \tilde{\mu_t}^{m-p-r} \wedge \tilde \omega^r \right| \\ \leq  C_r(t) \,\|\eta\|_{L^2\Omega^{p,0}(\tilde M,\tilde \omega )}\langle \Delta_{\overline{\partial},p,0}\eta,\eta\rangle^{\frac{1}{2}}_{L^2\Omega^{p,0}(\tilde M,\tilde\omega)}.
\end{multline*}
 
Since $(\tilde M,\tilde\omega)$ is complete, we can conclude that for each $\eta\in \mathcal{D}(\Delta_{\overline{\partial},p,0})$ it holds

\begin{multline}
\label{inequalityx}
\left| \int_{\tilde M} i^{p^2} \eta\wedge\overline{\eta}\wedge \tilde{\mu_t}^{m-p-r} \wedge \tilde\omega^r \right| \\
\leq  C_r(t)\, \|\eta\|_{L^2\Omega^{p,0}(\tilde M,\tilde\omega)}\langle \Delta_{\overline{\partial},p,0}\eta,\eta\rangle^{\frac{1}{2}}_{L^2\Omega^{p,0}(\tilde M,\tilde\omega)}.
\end{multline}

\medskip

With the notations of Proposition \ref{prop:normcontrol}, let us now choose $\lambda_0 = 1,$ $0 < \varepsilon < \varepsilon_0(1)$, and for $U_{\varepsilon}$ an open neighbourhood of $E$ with Lebesgue measure smaller then $\varepsilon$. Let $W$ be a relatively compact open set in $M\setminus E$ which contains $M \setminus U_{\varepsilon}$.

Since $\phi_t$ converges to $\phi_0$ in $C^{2,\alpha}(W)$, and $\mu' + i \partial \bar{\partial} \phi_0$ is a K\"ahler form on $M \setminus E$, then there exists a constant $C_1=C_1(\varepsilon)>0$ independent of $t$ such that $\tilde{\mu_t} + t\, \tilde{\omega} \geq C_1\, \tilde\omega$ on $\tilde M\setminus \tilde{U_{\varepsilon}}$, provided $t > 0$ is sufficiently small (here $\tilde{U_{\varepsilon}}$ is the preimage of $U_{\varepsilon}$ by the universal covering map). By Proposition \ref{prop:normcontrol} we have that
$
\int_{\tilde M}|\eta|^2_{\tilde\omega}\dvol_{\tilde\omega}\leq 2  \int_{{\tilde M\setminus \tilde U_{\varepsilon}}}|\eta|^2_{\tilde\omega}\dvol_{\tilde\omega},
$ 
for each $\eta\in \im(E(1))$. 

Therefore, for $\eta \in \im(E(1))$ we get, using Lemma \ref{lem:forms}:
\begin{equation}\label{inequalityz}
\begin{aligned}
\int_{\tilde M}|\eta|^2_{\tilde\omega}\dvol_{\tilde\omega} & \leq 2  \int_{{\tilde M\setminus \tilde U_{\varepsilon}}}|\eta|^2_{\tilde\omega}\dvol_{\tilde\omega} \\  
& = \frac{2}{(m-p)!} \int_{\tilde M\setminus \tilde U_{\varepsilon}} i^{p^2}  \eta \wedge \bar\eta \wedge \tilde\omega^{m-p} \\ 
& \leq \frac{2}{(m-p)!\,C_1^{m-p}} \int_{\tilde M\setminus \tilde U_{\varepsilon}} i^{p^2} \eta \wedge \bar\eta \wedge\bigl(\tilde\mu_t + t\,\tilde\omega\bigr)^{m-p} \\ 
& \leq  \frac{2}{(m-p)!\,C_1^{m-p}}   \int_{\tilde M} i^{p^2}  \eta \wedge \bar\eta \wedge\bigl(\tilde\mu_t + t\, \tilde\omega\bigr)^{m-p}.  
\end{aligned}  
\end{equation}

If we now choose $0 < t \leq 1$, then by inequality (\ref{inequalityx}), we have:
\begin{multline*}
\frac{2}{(m-p)!\,C_1^{m-p}}   \int_{\tilde M} i^{p^2}  \eta \wedge \bar\eta \wedge\bigl(\tilde\mu_t + t\, \tilde\omega\bigr)^{m-p} \\ 
= \left|\frac{2}{(m-p)!\,C_1^{m-p}}   \int_{\tilde M} i^{p^2}  \eta \wedge \bar\eta \wedge\bigl(\tilde\mu_t + t\, \tilde\omega\bigr)^{m-p}\right| \\
\leq \frac{2}{(m-p)!\,C_1^{m-p}} \sum_{r=0}^{m-p-1} t^r \binom{m-p}{r} \left| \int_{\tilde M} i^{p^2}  \eta \wedge \bar{\eta} \wedge \tilde{\mu_t}^{m-p-r} \wedge \tilde\omega^r \right| \\ 
+ \frac{2}{(m-p)!\,C_1^{m-p}}\, t^{m-p} \int_{\tilde M} i^{p^2} \eta \wedge \bar{\eta} \wedge \tilde\omega^{m-p} \\
\leq \frac{2}{(m-p)!\,C_1^{m-p}}\sum_{r=0}^{m-p-1}\binom{m-p}{r}t^r\,C_r(t) \\
\qquad\qquad\times \|\eta\|_{L^2\Omega^{p,0}(\tilde M,\tilde \omega )}\langle \Delta_{\overline{\partial},p,0}\eta,\eta\rangle^{\frac{1}{2}}_{L^2\Omega^{p,0}(\tilde M,\tilde\omega)}  \\ 
 \qquad\qquad\qquad + \frac{2}{(m-p)!\,C_1^{m-p}}\, t^{m-p}  \int_{\tilde M}i^{p^2} \eta \wedge \bar\eta \wedge \tilde\omega^{m-p} \\
 = \frac{2}{(m-p)!\,C_1^{m-p}}\sum_{r=0}^{m-p-1}\binom{m-p}{r}t^r\,C_r(t) \\
\qquad\qquad\times \|\eta\|_{L^2\Omega^{p,0}(\tilde M,\tilde \omega )}\langle \Delta_{\overline{\partial},p,0}\eta,\eta\rangle^{\frac{1}{2}}_{L^2\Omega^{p,0}(\tilde M,\tilde\omega)}  \\ 
 +\frac{2}{C_1^{m-p}}\,  t^{m-p} \int_{\tilde M}|\eta|^2_{\tilde\omega}\dvol_{\tilde\omega}.
\end{multline*}

So, by (\ref{inequalityz}), 
\begin{multline*}
\int_{\tilde M}|\eta|^2_{\tilde\omega}\dvol_{\tilde\omega} 
 \leq \frac{2}{(m-p)!\,C_1^{m-p}}\sum_{r=0}^{m-p-1}\binom{m-p}{r}t^r\,C_r(t) \\
\qquad\qquad\times\|\eta\|_{L^2\Omega^{p,0}(\tilde M,\tilde \omega )}\langle \Delta_{\overline{\partial},p,0}\eta,\eta\rangle^{\frac{1}{2}}_{L^2\Omega^{p,0}(\tilde M,\tilde\omega)} \\ 
 +\frac{2}{C_1^{m-p}}\,  t^{m-p} \int_{\tilde M}|\eta|^2_{\tilde\omega}\dvol_{\tilde\omega},
 \end{multline*} 
 \textsl{i.e.}
\begin{multline*}
\biggl(1 - \frac{2}{C_1^{m-p}}\,  t^{m-p}\biggr)\int_{\tilde M}|\eta|^2_{\tilde\omega}\dvol_{\tilde\omega} \\
\leq \frac{2}{(m-p)!\,C_1^{m-p}}\sum_{r=0}^{m-p-1}\binom{m-p}{r}t^r\,C_r(t) \\ \times
\|\eta\|_{L^2\Omega^{p,0}(\tilde M,\tilde \omega )}\langle \Delta_{\overline{\partial},p,0}\eta,\eta\rangle^{\frac{1}{2}}_{L^2\Omega^{p,0}(\tilde M,\tilde\omega)}.
\end{multline*}
Since $ 1 - 2t^{m-p}/C_1^{m-p}>0$ for  $t > 0 $ sufficiently small, we have that the function
$$ 
H(t) := \frac{\frac{2}{(m-p)!\,C_1^{m-p}}\sum_{r=0}^{m-p-1}\binom{m-p}{r}t^r\,C_r(t)}{1 - 2t^{m-p}/C_1^{m-p}}
$$ 
is non zero (in fact positive) for $t>0$ sufficiently small, so that we find in the end
 $$
 \langle \Delta_{\overline{\partial},p,0}\eta,\eta\rangle_{L^2\Omega^{p,0}(\tilde M,\tilde\omega)} \geq \frac{1}{H(t)^2}  \|\eta\|^2_{L^2\Omega^{p,0}(\tilde M,\tilde\omega)}. 
$$
Let us now consider any $\psi\in \mathcal{D}(\Delta_{\overline{\partial},p,0})$, let $\eta:=E(1)(\psi)$ and $\upsilon:=\psi-\eta$. Then,  
$$
\langle \eta,\upsilon\rangle_{L^2\Omega^{p,0}(\tilde M,\tilde\omega)} = \langle \eta,\Delta_{\overline{\partial},p,0}\upsilon\rangle_{L^2\Omega^{p,0}(\tilde M,\tilde\omega)} = 0 
$$ 
and 
$$ 
\langle\Delta_{\overline{\partial},p,0}\upsilon,\upsilon \rangle_{L^2\Omega^{p,0}(\tilde M,\tilde\omega)}  \geq \|\upsilon\|^2_{L^2\Omega^{p,0}(\tilde M,\tilde\omega)}.
$$
Therefore we have 
$$
\begin{aligned}
\langle\Delta_{\overline{\partial},p,0}\psi,\psi\rangle_{L^2\Omega^{p,0}(\tilde M,\tilde\omega)}& =\langle\Delta_{\overline{\partial},p,0}(\eta+\upsilon),\eta+\upsilon\rangle_{L^2\Omega^{p,0}(\tilde M,\tilde\omega)} \\
& = \langle\Delta_{\overline{\partial},p,0}\eta,\eta\rangle_{L^2\Omega^{p,0}(\tilde M,\tilde\omega)}+\langle\Delta_{\overline{\partial},p,0}\upsilon,\upsilon\rangle_{L^2\Omega^{p,0}(\tilde M,\tilde\omega)} \\
& \geq \frac{1}{H(t)^2}\|\eta\|^2_{L^2\Omega^{p,0}(\tilde M,\tilde\omega)}+\|\upsilon\|^2_{L^2\Omega^{p,0}(\tilde M,\tilde\omega)} \\
& \geq K(t)(\|\eta\|^2_{L^2\Omega^{p,0}(\tilde M,\tilde\omega)}+\|\upsilon\|^2_{L^2\Omega^{p,0}(\tilde M,\tilde\omega)}) \\
& =K(t)\|\psi\|^2_{L^2\Omega^{p,0}(\tilde M,\tilde\omega)},
\end{aligned}
$$
where 
$$
K(t):=\min\biggl\{\frac{1}{H(t)^2},1\biggr\}.
$$
In conclusion we showed that for small $t>0$ there exists a positive function $K(t)$ such that any $\eta\in \mathcal{D}(\Delta_{\overline{\partial},p,0})$ satisfies
$$
\langle \Delta_{\overline{\partial},p,0}\eta,\eta\rangle_{L^2\Omega^{p,0}(\tilde M,\tilde\omega)}\geq K(t)\|\eta\|^2_{L^2\Omega^{p,0}(\tilde M,\tilde\omega)}
$$ 
and this amounts to saying that $0\notin \sigma(\Delta_{\overline{\partial},p,0})$.

\medskip

We are left to show that $\Delta_{\overline{\partial},m,0}$ has closed range. Since $(\tilde M,\tilde\omega)$ is K\"ahler and complete 
$$
\Delta_{\overline{\partial},m-q,0}\colon L^2\Omega^{m-q,0}(\tilde M,\tilde\omega)\to L^2\Omega^{m-q,0}(\tilde M,\tilde\omega)
$$ 
is unitary equivalent to 
$$
\Delta_{\overline{\partial},m,q}\colon L^2\Omega^{m,q}(\tilde M,\tilde\omega)\to L^2\Omega^{m,q}(\tilde M,\tilde\omega)
$$
through the action of the Hodge star operator. Hence, $0\notin \sigma(\Delta_{\overline{\partial},m,q})$ with $q=1,\dots,m$. 

Let us now focus on the case $q=1$. Since $0\notin \sigma(\Delta_{\overline{\partial},m,1})$ we know that $\im(\Delta_{\overline{\partial},m,1})$ is closed in $L^2\Omega^{m,1}(\tilde M,\tilde\omega)$. Hence, 
$$
\overline{\partial}_{m,0}\colon L^2\Omega^{m,0}(\tilde M,\tilde\omega)\to L^2\Omega^{m,1}(\tilde M,\tilde\omega)
$$ 
has closed range and this in turn implies that its adjoint 
$$
\overline{\partial}_{m,0}^*\colon L^2\Omega^{m,1}(\tilde M,\tilde\omega)\to L^2\Omega^{m,0}(\tilde M,\tilde\omega)
$$ 
has closed range, too. Finally since both $\overline{\partial}_{m,0}$ and $\overline{\partial}_{m,0}^*$ have closed range we can conclude that 
$$
\overline{\partial}_{m,0}^*\circ \overline{\partial}_{m,0}\colon L^2\Omega^{m,0}(\tilde M,\tilde\omega)\to L^2\Omega^{m,0}(\tilde M,\tilde\omega)
$$ 
has closed range, \textsl{i.e.} $\im(\Delta_{\overline{\partial},m,0})$ is closed in $L^2\Omega^{m,0}(\tilde M,\tilde\omega)$ as required.
\end{proof} 

\begin{rem}
The above property amounts to saying that $0$ is at most an isolated eigenvalue of $\Delta_{\overline{\partial},m,0}$.
\end{rem}

\begin{rem}
It can be proved that if we moreover suppose that a class $[\mu]\in\mathcal W_M$ can be represented by a semi-positive form which is strictly positive almost everywhere then the conclusions of Theorem \ref{thm:spectralgap} hold also on the universal cover of $M$ with no need to pass to a modification first (cf. Version 1 of the present paper on the arXiv for more on this).
\end{rem}

In view of the remark above, it is natural to ask the following.

\begin{quest}
Is it possibile to prove Theorem \ref{thm:spectralgap} without passing to a modification first?
\end{quest}

\subsection{Nonvanishing}

In this section we aim to show that the space of $L^2$ holomorphic $m$-forms on $\tilde{M}$ is infinite dimensional (recall that we replaced $M'$ with $M$). Since the original argument of Gromov applies \textsl{verbatim}, we provide simply a sketch and we refer to \cite{Bal06,Eys97,Gro91} for details.

For each $p=0,...,m$ let 
$$
\overline{\eth}_p\colon L^2\Omega^{p,\bullet}(\tilde{M},\tilde{\omega})\to L^2\Omega^{p,\bullet}(\tilde{M},\tilde{\omega})
$$ 
be the unique closed (and hence self-adjoint) extension of 
$$
\overline{\partial}_p+\overline{\partial}_p^t\colon \Omega_c^{p,\bullet}(\tilde{M})\to \Omega^{p,\bullet}_c(\tilde{M}),
$$ 
with $\Omega_c^{p,\bullet}(\tilde{M}):=\bigoplus_{q=0}^m\Omega^{p,q}_c(\tilde{M})$ and $\overline{\partial}_p+\overline{\partial}^t_p$ the corresponding Dirac operator. 

Let $F$ be the trivial line bundle $\tilde{M}\times \mathbb{C}\rightarrow \tilde{M}$ endowed with the standard Hermitian metric and flat connection $\nabla_0$. 

Let $\mu\in \mathcal{W}_M$ be arbitrarily fixed and let $\alpha\in \Omega^1(\tilde{M})\cap L^{\infty}\Omega^1(\tilde{M},\tilde{\omega})$ be such that $d\alpha=\tilde{\mu}$. Given any $s>0$, let $\nabla^s$ be the connection on $F$ defined as 
$$
\nabla^s:=\nabla_0+is\alpha.
$$ 
Note that for each fixed $s$ there exists a (not necessarily standard) holomorphic structure on $F$ such that $\nabla^s$ becomes the corresponding Chern connection, see \cite[p. 191]{Eys97}. Let 
$$
(\overline{\partial}_p+\overline{\partial}_p^t)\otimes \nabla^s\colon C_c^{\infty}(\tilde{M},\Lambda^{p,\bullet}(\tilde{M})\otimes F)\to C_c^{\infty}(\tilde{M},\Lambda^{p,\bullet}(\tilde{M})\otimes F)
$$ 
be the first order differential operator obtained by twisting the Dirac operator $\overline{\partial}_p+\overline{\partial}_p^t$ with the connection $\nabla^s$, see \textsl{e.g.} \cite[p. 111]{Bal06}. Finally let us denote with 
\begin{equation}\label{L2twist}
\overline{D}_p^s\colon L^2(\tilde{M},\Lambda^{p,\bullet}(\tilde{M})\otimes F)\to L^2(\tilde{M},(\Lambda^{p,\bullet}(\tilde{M})\otimes F)
\end{equation}
the $L^2$-closure of $(\overline{\partial}_p+\overline{\partial}_p^t)\otimes \nabla^s$ and let $\overline{D}_p^{s,+}$ (resp. $\overline{D}_p^{s,-}$) be the operator induced by \eqref{L2twist} with respect to the splitting given by $(m,\bullet)$-forms with even/odd anti-holomorphic degree. 

Although $\overline{D}_p^{s,+}$ is not equivariant with respect to the action of $\pi_1(M)$, it is possible for each fixed $s>0$ to construct a group $\Gamma_s$ as a central extension of $\pi_1(M)$ with respect to $U(1)$, in such a way that $\Gamma_s$ acts on $L^2(\tilde{M},\Lambda^{p,\bullet}(\tilde{M})\otimes F)$, its action commutes with $\overline{D}_p^s$ and the $L^2$-index of $\overline{D}_p^{s,+}$ with respect to $\Gamma_s$ is computed by the following formula, see \cite[Prop. 9.1.1]{Eys97} or \cite[Th. 8.31]{Bal06}:
\begin{equation}\label{L2formula}
L^2-\operatorname{ind}_{\Gamma_s}(\overline{D}_p^{s,+})=\int_M\operatorname{Todd}(M)\wedge\operatorname{ch}(\Lambda^{p,0}(M))\wedge \operatorname{ch}(F),
\end{equation}
with $\operatorname{ch}(F)=\exp(-s\mu/2\pi)$. 

Since $\int_M\mu^m>0$, we deduce that \eqref{L2formula} is a polynomial function in $s$ with non trivial leading coefficient, and hence it has only isolated zeros, see \cite[p. 281]{Gro91}. 

In particular for all but a discrete subset of real values $s$ we have that $\ker(\overline{D}_p^s)\neq \{0\}$.

\begin{cor}\label{cor:existencel2}
Let $M$ be a weakly K\"ahler hyperbolic manifold, and $\tilde\omega$ be the lift on the universal cover $\tilde M$ of a K\"ahler form $\omega$ on $M$. 
Then, the space of holomorphic $m$-forms which are $L^2$ with respect to $\tilde\omega$ on $\tilde M$ is infinite dimensional and the operator 
$$ 
\Delta_{\overline{\partial},m,0}\colon L^2\Omega^{m,0}(\tilde{M},\tilde{\omega})\to L^2\Omega^{m,0}(\tilde{M},\tilde{\omega})
$$ 
has closed range. 

Moreover, for $0 \le p \le m - 1$, there are no non trivial $L^2$ (with respect to any Kähler metric coming from $M$) holomorphic $p$-forms on the universal cover $\tilde M\to M$.
\end{cor}

In particular, this corollary improve the statement of Theorem \ref{thm:spectralgap} about the closedness of the range of $\Delta_{\overline{\partial},m,0}$ which now holds on $M$ with no need to pass to the modification $M'$.

\begin{proof}
To begin with, let us observe that if $\nu\colon M'\to M$ is a birational morphism, then by \cite[Corollary 11.4]{Kol95} we have the equality of the $L^2$-Hodge numbers $h^{p,0}_{(2)}(M)=h^{p,0}_{(2)}(M')$ for $0\le p\le m$. The last assertion about the vanishing of holomorphic $p$-forms then follows directly from Theorem \ref{thm:spectralgap}.

From now on, let us work with the modification $M'$ of Theorem \ref{thm:spectralgap}, and let us call it $M$ for simplicity.
As recalled above, we know that $\ker(\overline{D}_p^s)\neq \{0\}$ except for a discrete subset of real values $s$. In particular there exists an $\epsilon>0$ such that $\ker(\overline{D}_p^s)\neq \{0\}$ for each $0<s<\epsilon$ and therefore, thanks to \cite[Prop. 7.1.2]{Eys97}, we can conclude that $0\in \sigma(\overline{\eth}_p)$, that is $0$ lies in the spectrum of $\overline{\eth}_p$. 

Let us now focus on the case $p=m$. Since 
\begin{multline}\label{directsum}
\overline{\eth}_m^2\colon L^2\Omega^{m,\bullet}(\tilde{M},\tilde{\omega})\to L^2\Omega^{m,\bullet}(\tilde{M},\tilde{\omega}) \\
=\bigoplus_{q=0}^m\Delta_{\overline{\partial},m,q}\colon L^2\Omega^{m,\bullet}(\tilde{M},\tilde{\omega})\to L^2\Omega^{m,\bullet}(\tilde{M},\tilde{\omega}),
\end{multline}
we know on the one hand that $0\in \sigma(\Delta_{\overline{\partial},m,q})$ for some $q=0,\dots,m$. 

On the other hand, as 
$$
\Delta_{\overline{\partial},m,q}\colon L^2\Omega^{m,q}(\tilde{M},\tilde{\omega})\to L^2\Omega^{m,q}(\tilde{M},\tilde{\omega})
$$ 
and 
$$
\Delta_{\overline{\partial},m-q,0}\colon L^2\Omega^{m-q,0}(\tilde{M},\tilde{\omega})\to L^2\Omega^{m-q,0}(\tilde{M},\tilde{\omega})
$$ 
are unitary equivalent through the action of the Hodge star operator, we deduce from Theorem \ref{thm:spectralgap} that $0\notin  \sigma(\Delta_{\overline{\partial},m,q})$ for each $q=1,\dots,m$. Hence $0\in \sigma(\Delta_{\overline{\partial},m,0})$.  

Furthermore, given that $\Delta_{\overline{\partial},m,0}$ has closed range we can conclude that $0$ is actually an eigenvalue of $\Delta_{\overline{\partial},m,0}$, \textsl{i.e.} $\ker(\Delta_{\overline{\partial},m,0})\neq \{0\}$. 

As previously recalled, in $L^2\Omega^{m,0}(\tilde{M},\tilde{\omega})$ we have $\ker(\Delta_{\overline{\partial},m,0})=\ker(\overline{\partial}_{m,0})$. Hence, we know that on $\tilde{M}$ the space of $L^2$ holomorphic $m$-forms is not trivial. 

Finally since this space is preserved by the action of $\pi_1(M)$ we can conclude that it is infinite dimensional, see \cite[Lemma 15.10]{Roe98}.

\medskip

Next we show that closedness of the range of $\Delta_{\overline{\partial},m,0}$ holds on $M$. Since $(\tilde{M}, \tilde{\omega})$ is complete, in order to show that $\Delta_{\overline{\partial},m,0}$ has closed range it suffices to prove that both 
$$
\overline{\partial}_{m,0}\colon L^2\Omega^{m,0}(\tilde{M},\tilde{\omega})\to L^2\Omega^{m,1}(\tilde{M},\tilde{\omega})
$$ 
and its adjoint
$$
\overline{\partial}_{m,0}^*\colon L^2\Omega^{m,1}(\tilde{M},\tilde{\omega})\to L^2\Omega^{m,0}(\tilde{M},\tilde{\omega})
$$ 
have closed range. Moreover, by the fact that these two operators are one the adjoint of the other it is enough to show only that the first one above has closed range. Let $\nu\colon M'\to M$ be the modification introduced in Theorem \ref{thm:spectralgap} endowed with a K\"ahler metric $\omega'$. Let $X\subset M'$ and $Y\subset M$ be analytic subsets such that $\nu|_A\colon A\to B$ is a biholomorphism, with $B:=M\setminus Y$ and $A:=M'\setminus X$. Let 
$$
\pi'\colon(\tilde{M}',\tilde{\omega}')\to (M',\omega')
\quad\textrm{and}\quad
\pi\colon(\tilde{M},\tilde{\omega})\to (M,\omega)
$$ 
be the corresponding universal coverings and let $\tilde{\nu}\colon\tilde{M}'\to \tilde{M}$ be a lift of $\nu$. 

Clearly $\tilde{\nu}\colon\tilde{A}\rightarrow \tilde{B}$ is a biholomorphism, with $\tilde{A}:=(\pi')^{-1}(A)$ and $\tilde{B}:=\pi^{-1}(B)$. Moreover, we point out that 
\begin{equation}\label{pull1}
\nu^*\colon L^2\Omega^{m,0}(\tilde{M},\tilde{\omega})\to L^2\Omega^{m,0}(\tilde{M}',\tilde{\omega}')
\end{equation}
is an isometry of Hilbert spaces and 
\begin{equation}\label{pull2}
\nu^*\colon L^2\Omega^{m,q}(\tilde{M},\tilde{\omega})\to L^2\Omega^{m,q}(\tilde{M}',\tilde{\omega}')
\end{equation}
is injective and continuous for each $q=1,\dots,m$. These claims follow easily \textsl{e.g.} by the computations carried out in \cite[Equations (7) and (10)]{Bei19} and the fact that both $\tilde{M}'\setminus \tilde{A}$ and $\tilde{M}\setminus \tilde{B}$ have measure zero. 

Let now $\eta\in L^2\Omega^{m,1}(\tilde{M},\tilde{\omega})$ be a form lying in the closure of the range of $\overline{\partial}_{m,0}\colon L^2\Omega^{m,0}(\tilde{M},\tilde{\omega})\to L^2\Omega^{m,1}(\tilde{M},\tilde{\omega})$. Since $(\tilde{M},\tilde{\omega})$ is complete, there exists a sequence $\{\phi_j\}\subset\Omega^{m,0}_c(\tilde{M})$ such that $\overline{\partial}_{m,0}\phi_j\to \eta$ in $L^2\Omega^{m,1}(\tilde{M},\tilde{\omega})$ as $j\rightarrow \infty$. By \eqref{pull1} and \eqref{pull2}, the sequence $\{\nu^*\phi_j\}$ is made of smooth forms, it is contained in the domain of $\overline{\partial}_{m,0}\colon L^2\Omega^{m,0}(\tilde{M}',\tilde{\omega}')\to L^2\Omega^{m,1}(\tilde{M}',\tilde{\omega}')$ and $\overline{\partial}_{m,0}(\nu^*\phi_j)\to \nu^*\eta$ in $L^2\Omega^{m,1}(\tilde{M}',\tilde{\omega}')$ as $j\rightarrow \infty$.

Since $\bar\partial_{m,0}\colon L^2\Omega^{m,0}(\tilde{M}',\tilde{\omega}')\to L^2\Omega^{m,1}(\tilde{M}',\tilde{\omega}')$ has closed range, there exists $\chi\in L^2\Omega^{m,0}(\tilde M',\tilde\omega')$ such that $\bar\partial_{m,0}\chi=\nu^*\eta$. We claim that $\psi:=(\nu^*)^{-1}\chi$ lies in the domain of $\bar\partial_{m,0}\colon L^2\Omega^{m,0}(\tilde{M},\tilde{\omega})\to L^2\Omega^{m,1}(\tilde{M},\tilde{\omega})$ and that $\bar\partial_{m,0}\psi=\eta$. Thanks to \cite[Proposition 5.6]{Bei19b} we know that the operator $\bar\partial_{m,0}\colon L^2\Omega^{m,0}(\tilde{M},\tilde{\omega})\to L^2\Omega^{m,1}(\tilde{M},\tilde{\omega})$ equals the $L^2$ distributional extension
$$
\bar\partial_{m,0,\textrm{max}}\colon L^2\Omega^{m,0}(\tilde{B},\tilde{\omega}|_{\tilde B})\to L^2\Omega^{m,1}(\tilde{B},\tilde{\omega}|_{\tilde B})
$$
of $\bar\partial_{m,0}\colon \Omega^{m,0}_c(\tilde{B})\to \Omega^{m,1}_c(\tilde{B})$.

Hence, in order to verify the claim above, it is enough to check that 
$$
(-1)^{m-1}\int_{\tilde M}\psi\wedge\bar\partial_{0,m-1}\varphi=\int_{\tilde M}\eta\wedge\varphi,\quad\forall \varphi\in\Omega_c^{0,m-1}(\tilde B).
$$
We have
$$
\begin{aligned}
(-1)^{m-1}\int_{\tilde M}\psi\wedge\bar\partial_{0,m-1}\varphi &= (-1)^{m-1}\int_{\tilde B}\psi\wedge\bar\partial_{0,m-1}\varphi \\
&=(-1)^{m-1}\int_{\tilde A}\underbrace{\nu^*\psi}_{=\chi}\wedge\bar\partial_{0,m-1}\nu^*\varphi \\
&=\int_{\tilde A}\nu^*\eta\wedge\nu^*\varphi =\int_{\tilde B}\eta\wedge\varphi,
\end{aligned}
$$
and the claim is proved.

We can thus conclude that 
$$
\overline{\partial}_{m,0}\colon L^2\Omega^{m,0}(\tilde{M},\tilde{\omega})\to L^2\Omega^{m,1}(\tilde{M},\tilde{\omega})
$$ 
has closed range and so that 
$$ 
\Delta_{\overline{\partial},m,0}\colon L^2\Omega^{m,0}(\tilde{M},\tilde{\omega})\to L^2\Omega^{m,0}(\tilde{M},\tilde{\omega})
$$ 
has closed range, as well. 
\end{proof}

\section{Geometric Consequences}

In this section we derive three type of geometric consequences of our weaker notion of Kähler hyperbolicity. Namely, we deduce several statements of general typeness, but also index-type consequences and, to finish with, geometric consequences about the distribution of entire curves and Kobayashi hyperbolicity properties. In the last part we shall derive stronger consequences in the case of surfaces.

\subsection{General typeness, index, and topology}\label{subsect:generaltype}

We already know that a Kähler hyperbolic manifold is of general type. We now see how the weaker notion of weakly Kähler hyperbolicity gives in fact the same consequence.

\begin{thm}\label{thm:weaklygeneral}
Let $M$ be a weakly Kähler hyperbolic manifold. Then, $M$ is projective and of general type and moreover the Euler characteristic $\chi(K_M)$ is strictly positive.
\end{thm}

The last part of the statement above extends, for $p=m$, Gromov's computation for Kähler hyperbolic manifolds \cite[Theorem 0.4.A]{Gro91} to weakly Kähler hyperbolic manifolds. 

\begin{proof}
By Corollary \ref{cor:existencel2}, there exists on $\tilde M$ a non-zero $L^2$ holomorphic top form. This implies that $K_M$ is big, thanks to \cite[13.10 Corollary]{Kol95}, since $M$ has generically large fundamental group. This means by definition that $M$ is of general type. Since $K_M$ is a big line bundle, the manifold $M$ is Moishezon, and being also K\"ahler, it is projective by Moishezon's theorem.

The last part follows from Theorem \ref{thm:gammaHodge}, Theorem \ref{thm:spectralgap}, Corollary \ref{cor:existencel2}, and the birational invariance \cite[Corollary 11.4]{Kol95} of $L^2$-Hodge numbers in bidegree $(m,q)$, being $M$ Kähler.
\end{proof}

Thanks to \cite[4.4 Corollary]{ABR92}, it is known that the universal cover $\tilde M$ of a compact Kähler manifold $M$ has at most one end. 
Here, we give a short proof of this fact in the case where $M$ is a weakly Kähler hyperbolic manifold.

\begin{prop}
Let $M$ be a weakly Kähler hyperbolic manifold and let $p\colon\tilde{M}\rightarrow M$ be its universal covering. Then, $\tilde{M}$ has only one end, and
$$
H^1_c(\tilde{M},\mathbb{Z})=H_{2m-1}(\tilde{M},\mathbb{Z})=\{0\}.
$$ 
\end{prop}

\begin{proof}
Fix a Kähler metric $\omega$ on $M$ and let $\tilde\omega$ be its lifting to the universal cover. We follow \cite[Section 5.1]{CP04} (see also \cite[Section 4]{LW01}, cf. in particular  \cite[Proposition 5.1 and 5.2]{CP04}. First of all we want to show that the first de Rham cohomology group with compact support of $\tilde{M}$, denoted here with $H^1_c(\tilde{M})$, is trivial. Let us consider the $L^2$-de Rham cohomology group of degree one: 
$$
H^{1}_{(2)}(\tilde{M},\tilde{\omega}):=\ker(d_1)/\text{im}(d_0),
$$ 
where $d_0\colon L^2(\tilde{M},\tilde{\omega})\to L^2\Omega^1(\tilde{M},\tilde{\omega})$ and $d_1\colon L^2\Omega^1(\tilde{M},\tilde{\omega})\to L^2\Omega^2(\tilde{M},\tilde{\omega})$ are the $L^2$-closed extensions of $d_0\colon C^{\infty}_c(\tilde{M})\to \Omega_c^{1}(\tilde{M})$ and $d_1\colon\Omega^1_c(\tilde{M})\to \Omega_c^{2}(\tilde{M})$, respectively. 

Since $(\tilde M, \tilde{\omega})$ has bounded geometry, it is not difficult to see that given any compact subset $K\subset \tilde{M}$, each non-compact connected component of $\tilde{M}\setminus K$ has infinite volume with respect to $\tilde{\omega}$. Therefore, using this latter remark, it is easy to deduce that the inclusion $\Omega_c^1(\tilde{M})\hookrightarrow L^2\Omega^1(\tilde{M},\tilde{\omega})$ induces an injective map $H^1_c(\tilde{M})\rightarrow H^1_{(2)}(\tilde{M},\tilde{\omega})$. 

On the other hand, thanks to Theorem \ref{thm:spectralgap}, birational invariance, $L^2$ Hodge decomposition, begin $(\tilde{M},\tilde{\omega})$ K\"ahler and complete, we know that $\mathcal{H}^1_{(2)}(\tilde{M},\tilde{\omega})=\{0\}$, that is the space of $L^2$-harmonic $1$-forms on $(\tilde{M},\tilde{\omega})$ is trivial. 
Moreover, since $0\notin \sigma(\Delta_0)$, we also know that $H^1_{(2)}(\tilde{M},\tilde{\omega})\simeq \mathcal{H}^1_{(2)}(\tilde{M},\tilde{\omega})$. Hence, we can conclude that $H^1_c(\tilde{M})=\{0\}$, as desired. 

Next, let us recall how the vanishing of $H^1_c(\tilde{M})$ implies that $\tilde{M}$ has only one end. If there were a compact subset $K\subset \tilde{M}$ such that $\tilde{M}\setminus K$ has two non-compact connected components $X$ and $Y$, then we could find a relatively compact open subset $U\supset K$ and $f\in C^{\infty}(\tilde{M})$ such that $\operatorname{Supp}(f)\subset \overline{U}\cup X\cup Y$, $f|_{X\setminus (\overline{U}\cap X)}=c_1$ and $f|_{Y\setminus (Y\cap \overline{U})}=c_2$ with $c_1,c_2\in \mathbb{R}$ and $c_1\neq c_2$.  At this point, it is straightforward to check that $0\neq [d_0f]\in H^1_c(\tilde{M})$, thus contradicting the fact that $H^1_c(\tilde{M})=\{0\}$. 

Next, we observe that since $\tilde{M}$ is simply connected, we know that we have $H^1(\tilde{M},\mathbb{Z})=\{0\}$, as well. Finally, since $\tilde{M}$ as only one end, we have an injection $H^1_c(\tilde{M},\mathbb{Z})\hookrightarrow H^1(\tilde{M},\mathbb{Z})$, and thus we can conclude that $H^1_c(\tilde{M},\mathbb{Z})=\{0\}$. Finally, $H_{2m-1}(\tilde{M},\mathbb{Z})=\{0\}$ by Poincar\'e duality.
\end{proof}

\subsection{Applications to Lang's conjecture and Kobayashi hyperbolicity}\label{subsect:Lang}

The first thing we want to treat here is the distribution of entire curves in a weakly Kähler hyperbolic manifold. We begin with the following basic lemma.

\begin{lem}\label{lem:current}
Let $(M,\mu)$ be a weakly K\"ahler hyperbolic manifold, and fix a Kähler form $\omega$ on M. If $f\colon\mathbb{C} \to M$ is an entire curve, and $\Phi$ is an Ahlfors current associated to $f\colon\mathbb C\to (M,\omega)$, then $\Phi(\mu) = 0$.   
\end{lem}

\begin{proof}
Let $\tilde{f}\colon\mathbb{C} \to \tilde{M}$ be a lifting of $f$. Assume that $\tilde{\mu} = d \eta $ with 
$|\eta|_{\tilde{\omega}}\leq C$. Then for any vector $v\in T_\mathbb C$ we have:
$$ 
|\tilde{f}^*\eta(v)| = |\eta(\tilde{f}_*(v)| \leq C |\tilde{f}_*(v)|_{\tilde{\omega}} = C |f_*(v)|_{\omega}.
$$
Then, by Stokes theorem,
$$  
\begin{aligned}
\left |\int_{\mathbb D_t} f^*\mu \right|& = \left|\int_{\mathbb D_t} \tilde{f}^*d \eta\right| \\
& = \left|\int_{\partial\mathbb D_t}  \tilde{f}^*\eta\right| \leq C\,L_{\omega}(\partial\mathbb D_t).
\end{aligned}
$$
If we divide the above inequality by $t$ and integrate over $(0,r)$, we find:
$$
|T_{f,r}(\mu) |\leq C\,S_{f,r}(\omega).
$$ 
Thus, if we divide by $T_{f,r}(\omega)$ we obtain:
$$ 
0 \leq \frac{|T_{f,r}(\mu)|}{T_{f,r}(\omega)} \leq  C\, \frac{S_{f,r}(\omega)}{T_{f,r}(\omega)}. 
$$  
We finally choose an admissible sequence $\{r_k\}$ in the above inequality, and let $k$ goes to $+\infty$, to obtain $\Phi(\mu) = 0$. 
\end{proof}

Using this lemma we can prove the theorem below.

\begin{thm}\label{thm:GG}
If $M$ is a weakly K\"ahler hyperbolic manifold, then the image of any entire curve in $M$ is contained in $Z_M$.  

In particular if $Z_M$ has dimension $1$ with no rational or elliptic components, then $M$ is Kobayashi hyperbolic.  
\end{thm}

Thus, in a weakly Kähler hyperbolic manifold we cannot have Zariski dense entire curves or an entire curve passing through a general point.

Since weakly Kähler hyperbolic implies general type, Theorem \ref{thm:GG} then gives a confirmation of the (stronger version of the) Green--Griffiths conjecture for this particular class of manifolds of general type.

\begin{proof}
Let $\mu \in\mathcal{W}_M$, and let $f\colon\mathbb{C} \to M$ be an entire curve. By \cite[Theorem 3.17, (ii)]{Bou04}  there exists a K\"ahler current $R$ in the cohomology class of $\mu$ which is smooth on $M \setminus E_{nK}([\mu]).$ As in the proof of Theorem \ref{thm:spectralgap}, we consider a modification $\tau\colon\hat{M} \to M $ such that 
\begin{itemize}
\item the pull-back $\tau^*R$ is cohomologous to the sum of a Kähler form $\omega$ and a real effective divisor $D = \sum_{1 \leq j \leq N} \lambda_j\, E_j$ on $\hat{M}$;
\item the modification $\tau$ is a biholomorphism between $\hat{M} \setminus\operatorname{Supp}(D)$ and $M \setminus E_{nK}([\mu])$;
\item each irreducible divisor $E_j$  is mapped by $\tau$ into $E_{nK}([\mu])$.   
\end{itemize}

Observe that, by (ii) of Proposition \ref{prop:genfinite}, the manifold $\hat{M}$ is weakly K\"ahler hyperbolic, since $\tau^*\mu\in\mathcal W_{\hat{M}}$.  
We want to show that $f(\mathbb{C}) \subseteq E_{nK}([\mu])$. Assume this is not the case, then the map $f$ can be lifted to an entire curve $\hat{f}\colon \mathbb{C} \to \hat{M}.$ Let $\Phi$ be an Ahlfors current associated to $\hat{f}\colon\mathbb C\to (\hat M,\omega)$. Then, by Lemma \ref{lem:current}, we have
$$ 
0 =[\Phi]\cdot[\tau^*\mu]  = [\Phi]\cdot[\omega] + \sum_j \lambda_j\, [\Phi]\cdot[E_j].
$$
Since $\Phi$ is an Ahlfors current associated to $\hat{f}\colon\mathbb C\to (\hat M,\omega)$, we have $[\Phi]\cdot[\omega]=1$, and therefore there exists $j_0$, $1 \leq j_0 \leq N$, such that $ [\Phi]\cdot[ E_{j_0}] < 0$. 
By \cite[Lemme 1]{Bru99}, it follows that $\hat{f}(\mathbb{C}) \subseteq E_{j_0}$. Since $\tau(E_{j_0}) \subseteq E_{nK}([\mu])$, then  
$f(\mathbb{C}) \subseteq E_{nK}([\mu])$ and we get a contraddiction. 

Finally, if $Z_M$ has dimension $1$ with no rational or elliptic components, then there exists no entire curves in $M$, and $M$ is Kobayashi hyperbolic by Brody's criterion.   
\end{proof}    

We can now come to the proof of Lang's conjecture for Kähler hyperbolic manifolds.

\begin{thm}\label{thm:langkhyp}
Let $M$ be a K\"ahler hyperbolic manifold, and let $X\subseteq M$ be a possibly singular closed subvariety. 

Then, any desingularization $\hat X\to X$ is of general type, \textsl{i.e.} $X$ is of general type.
\end{thm}
\begin{proof}
Being K\"ahler hyperbolic is a hereditary property for closed complex submanifolds. We already know that a K\"ahler hyperbolic manifold is of general type, therefore we are left to prove that singular subvarieties of $M$ are of general type. 
Given $X$ such a subvariety,  we may assume $X$ to be irreducible, and let $\nu\colon \hat X\to X$ be a desingularization. Then $\hat X$ is a compact connected K\"ahler manifold, and the map $\nu\colon \hat X \to M$ is a generically one-to-one map, so by (i) of Proposition \ref{prop:genfinite} we know that $\hat X$ is semi-K\"ahler hyperbolic. We thus conclude by Theorem \ref{thm:weaklygeneral} that $\hat X$ is of general type as desired.
\end{proof}

We can also prove a statement in the opposite direction, \textsl{i.e.} that a manifold of general type together with all of its subvarieties is Kobayashi hyperbolic. 

\begin{prop}\label{prop:Kob}
Let $M$ be a weakly K\"ahler hyperbolic manifold, and assume that any positive dimensional irreducible subvariety $X$ of $M$ has a  weakly K\"ahler hyperbolic desingularization. Then, $M$ is Kobayashi hyperbolic.
\end{prop}

\begin{proof}
Assume by contradiction that $M$ is not Kobayashi hyperbolic, and let $f: \mathbb{C} \to M$ be an entire curve. Let $X$ be the Zariski closure of $f(\mathbb C)$ and let $\tau\colon\hat{X} \to X$ be a weakly Kähler hyperbolic desingularization of $X$. Let $E\subset\hat X$ be the exceptional locus of $\tau$. Since $f(\mathbb{C})$ is Zariski dense in $X$, it is not contained in $\tau(E)$. Then, there exists a lifting $\hat{f}: \mathbb{C}  \to \hat{X}$ of $f$. Since $\hat{X}$ is weakly K\"ahler hyperbolic then by Proposition \ref{thm:GG} $\hat{f}(\mathbb{C}) \subseteq Z_{\hat X}$, so that $f(\mathbb{C}) \subseteq \tau(Z_{\hat X})$. This is impossible since $f(\mathbb{C})$ is Zariski dense in $X$ and $X$ is irreducible.
\end{proof}

\subsection{Surfaces}

To finish with, we restrict our attention to the case of surfaces. Indeed, in dimension two it is possible to get stronger statements. First, we generalize the exact analogous of \cite[Theorem 0.4.A, Remark 0.4.B]{Gro91} in the weakly Kähler hyperbolic case in the next theorem and corollary.

\begin{thm} 
Let $S$ be a weakly K\"ahler hyperbolic surface. Let $\omega$ be an arbitrarily fixed K\"ahler form on $S$. Let $\pi\colon(\tilde{S},\tilde{\omega})\to (S,\omega)$ be the K\"ahler universal covering of $(S,\omega)$. Finally, let 
$$
\Delta_{\overline{\partial},p,q}\colon L^2\Omega^{p,q}(\tilde{S},\tilde{\omega})\to L^2\Omega^{p,q}(\tilde{S},\tilde{\omega})\quad \mathrm{and}\quad \Delta_{k}\colon L^2\Omega^{k}(\tilde{S},\tilde{\omega})\to L^2\Omega^{k}(\tilde{S},\tilde{\omega})
$$
be the $L^2$-closures of $\Delta_{\overline{\partial},p,q}\colon\Omega_c^{p,q}(\tilde{S})\to \Omega_c^{p,q}(\tilde{S})$ and $\Delta_{k}\colon\Omega_c^{k}(\tilde{S})\to \Omega_c^{k}(\tilde{S})$, respectively. 

Then, we have the following two statements.

\begin{enumerate}
\item If $p+q\neq 2$, then $0\notin \sigma(\Delta_{\overline{\partial},p,q})$. If $p+q=2$, then $\ker(\Delta_{\overline{\partial},p,q})$ is infinite dimensional and $\mathrm{im}(\Delta_{\overline{\partial},p,q})$ is closed.
\item If $k\neq 2$, then $0\notin \sigma(\Delta_{k})$. If $k=2$, then $\ker(\Delta_{k})$ is infinite dimensional and $\mathrm{im}(\Delta_{k})$ is closed.
\end{enumerate}
\end{thm}

\begin{proof}
Thanks to Corollary \ref{cor:zeronotspect} we know that $0\notin \sigma(\Delta_{0})$. Since $(\tilde{S},\tilde{\omega})$ is K\"ahler and complete we know that 
$$
\Delta_{\overline{\partial},0,0}\colon L^2\Omega^{0,0}(\tilde{S},\tilde{\omega})\to L^2\Omega^{0,0}(\tilde{S},\tilde{\omega})
$$ 
equals 
$$
\frac{1}{2}\Delta_{0}\colon L^2\Omega^{0}(\tilde{S},\tilde{\omega})\to L^2\Omega^{0}(\tilde{S},\tilde{\omega}).
$$ 
Therefore $0\notin \sigma(\Delta_{\overline{\partial},0,0})$ and, using the Hodge star operator, we conclude also that $0\notin \sigma(\Delta_{\overline{\partial},2,2})$. 

As a next step we want to show that $0\notin \sigma(\Delta_{\overline{\partial},2,1})$. This is equivalent to proving that $\ker(\Delta_{\overline{\partial},2,1})=\{0\}$ and that 
$$
\Delta_{\overline{\partial},2,1}\colon L^2\Omega^{2,1}(\tilde{S},\tilde{\omega})\to L^2\Omega^{2,1}(\tilde{S},\tilde{\omega})
$$ 
has closed range. By Theorem \ref{thm:spectralgap} we know that
$$
\Delta_{\overline{\partial},1,0}\colon L^2\Omega^{1,0}(\tilde{S}',\tilde{\omega}')\to L^2\Omega^{1,0}(\tilde{S}',\tilde{\omega}')
$$ 
has trivial kernel. Thanks to \cite[11.4 Corollary]{Kol95}, the fact that $(\tilde{S},\tilde{\omega})$ is K\"ahler and complete, and using the Hodge star operator we can conclude that 
$$
\Delta_{\overline{\partial},2,1}\colon L^2\Omega^{2,1}(\tilde{S},\tilde{\omega})\to L^2\Omega^{2,1}(\tilde{S},\tilde{\omega})
$$ 
has trivial kernel, too. In order to show that $\im(\Delta_{\overline{\partial},2,1})$ is closed it is enough to show that both 
$$
\overline{\partial}_{2,0}\colon L^2\Omega^{2,0}(\tilde{S},\tilde{\omega})\to L^2\Omega^{2,1}(\tilde{S},\tilde{\omega})
$$ 
and 
$$
\overline{\partial}_{2,1}^*\colon L^2\Omega^{2,2}(\tilde{S},\tilde{\omega})\to L^2\Omega^{2,1}(\tilde{S},\tilde{\omega})
$$ 
have closed range; this latter implication follows easily by using the $L^2$-Hodge--Kodaira decomposition. Note that by the proof of Corollary \ref{cor:existencel2} we already know that $\overline{\partial}_{2,0}\colon L^2\Omega^{2,0}(\tilde{S},\tilde{\omega})\to L^2\Omega^{2,1}(\tilde{S},\tilde{\omega})$ has closed range.

Moreover, as a consequence of the fact that $0\notin \sigma(\Delta_{\overline{\partial},2,2})$, we obtain that both 
$$
\overline{\partial}_{2,1}\colon L^2\Omega^{2,1}(\tilde{S},\tilde{\omega})\to L^2\Omega^{2,2}(\tilde{S},\tilde{\omega})
$$ and
$$
\overline{\partial}_{2,1}^*\colon L^2\Omega^{2,2}(\tilde{S},\tilde{\omega})\to L^2\Omega^{2,1}(\tilde{S},\tilde{\omega})
$$ 
have closed range. Thus, we can conclude that $0\notin \sigma(\Delta_{\overline{\partial},2,1})$. 

Now, by using conjugation, the Hodge star operator, the fact that $(\tilde{S},\tilde{\omega})$ is K\"ahler and complete, and the $L^2$-Hodge--Kodaira decomposition, we obtain that $0\notin \sigma(\Delta_{\overline{\partial},p,q})$, with $p+q\neq 2$, and that $\im(\Delta_{\overline{\partial},p,q})$ is closed for each $p,q\in\{0,1,2\}$. Since we have already observed in the proof of Corollary \ref{cor:existencel2} that $0\in \sigma(\overline{\eth}_p)$ for each $p\in \{0,1,2\}$, we can conclude as in Corollary \ref{cor:existencel2} that $\Delta_{\overline{\partial},p,2-p}\colon L^2\Omega^{p,2-p}(\tilde{S},\tilde{\omega})\to L^2\Omega^{p,2-p}(\tilde{S},\tilde{\omega})$ has infinite dimensional kernel for each $p=0,1,2$. 
The proof of (1) is thus complete. 

Concerning (2), we observe that since $(\tilde{S},\tilde{\omega})$ is K\"ahler and complete we have an equality between 
$$
2\Delta_{\overline{\partial},p,q}\colon L^2\Omega^{p,q}(\tilde{S},\tilde{\omega})\to L^2\Omega^{p,q}(\tilde{S},\tilde{\omega})
$$ 
and 
$$
\Delta_{p+q}|_{L^2\Omega^{p,q}(\tilde{S},\tilde{\omega})}\colon L^2\Omega^{p,q}(\tilde{S},\tilde{\omega})\to L^2\Omega^{p,q}(\tilde{S},\tilde{\omega})
$$ 
for each $p,q\in\{0,1,2\}$. Now the conclusion follows immediately by the first point.
\end{proof}

\begin{rem}
Note that the above proof does not work to show that $\Delta_{\overline{\partial},p,q}$ has closed range for arbitrary $(p,q)$ because we used in a crucial way \eqref{pull1} and \eqref{pull2} and these properties in general do not hold when $(p,q)$ differs from $(s,0)$ and $(s,q)$ with $q>0$, respectively.
\end{rem}

\begin{cor}
Let $S$ be a weakly K\"ahler hyperbolic surface. Then, for each $0\leq p\leq 2$, the Euler characteristic $\chi(S,\Omega^p_S)$ does not vanish and we have 
$$
\operatorname{sign}\chi(S,\Omega^p_S)=(-1)^p.
$$ 
Moreover $\chi_\textrm{top}(S)>0$ and the $L^2$-Hodge and Betti numbers, $h^{p,q}_{(2)}(S)$ and $h^{p+q}_{(2)}(S)$, vanish if and only if $p+q\neq 2= \dim S$.
\end{cor}
\begin{proof}
This follows at once by the above theorem and Atiyah's $L^2$-index theorem.

Observe anyway, that our surface $S$ is of general type and hence the claims about the various Euler characteristic then follow, since they hold true in general for surfaces of general type. Indeed, we immediately get by \cite[(2.4) Proposition]{BHPVdV04} that $c_2(S)=\chi_\textrm{top}(S)>0$, and by \cite[Theorem X.4]{Bea96} that $\chi(S,\mathcal O_S)>0$. But then, by duality, $\chi(S,K_S)=\chi(S,\mathcal O_S)>0$.
Since 
$$
\chi_\textrm{top}(S)=\chi(S,\mathcal{O}_S)-\chi(S,\Omega^1_S)+\chi(S,K_S),
$$ 
we obtain 
$$
\begin{aligned}
\chi(S,\Omega^1_S) & = 2\,\chi(S,\mathcal O_S)-\chi_\textrm{top}(S) \\
& = \frac{c_1(S)^2+c_2(S)}{6}-c_2(S),
\end{aligned}
$$
thanks to Noether's formula and the  Chern--Gauss--Bonnet Theorem. But then, since $c_2(S)>0$, using the Bogomolov--Miyaoka--Yau inequality we find
$$
\begin{aligned}
6\,\chi(S,\Omega_S^1) & = c_1(S)^2-5\,c_2(S) \\
& < c_1(S)^2-3\, c_2(S)\le 0.
\end{aligned}
$$
\end{proof}

Finally, Proposition \ref{thm:GG} has the following stronger formulation in the case of surfaces, which in turn leads to a confirmation of the strongest possibile version of the Green--Griffiths conjecture for this particular class of surfaces of general type.

\begin{prop}\label{prop:GGdim2}
Let $S$ be a weakly K\"ahler hyperbolic surface. Then, the image of any entire curve in $S$ is contained in $Z_S$, and hence in a rational or elliptic irreducible component of $Z_S$.   

In particular, there exists only a finite number of rational or elliptic curves in $S$, and moreover the Gram intersection matrix of the family of all irreducible rational and elliptic curves is negative definite.  
\end{prop}
 
 \begin{proof}

By (iii) of Proposition \ref{prop:genfinite}, any rational or elliptic curve lies in $Z_S$, hence it is an irreducible component of $Z_S$. In particular, there is only a finite number 
$\{C_1,C_2, \dots,C_N\}$ of rational and elliptic curves. 

By Theorem \ref{thm:GG}, the image of any entire curve is contained in $Z_S$, hence in one of the curves $C_i$, $1 \leq i \leq N$. Given $\mu \in \mathcal{W}_S$, we have that any rational or elliptic curve is in $Z_S\subseteq E_{nK}([\mu])$. Then, by \cite[Propositions 3.12 and Theorem 4.5]{Bou04} the Gram intersection matrix $(C_i\cdot C_j)_{i,j=1,\dots,N}$ is negative definite.
 \end{proof}

\begin{cor}\label{cor:relative}
Let $S$ be a weakly K\"ahler hyperbolic surface, and let $C$ be the union of all its rational and elliptic curves. If we contract each connected component of $C$ to a point, we get a (possibly singular) Kobayashi hyperbolic surface. 

In particular, two distinct points of $S$ have Kobayashi distance $0$ in $S$ if and only if they both belong to the same connected component of $C$, and hence $S$ is Kobayashi hyperbolic modulo $C$.  
\end{cor}
 
 \begin{proof}
The proof is on the line of the proof of \cite[Proposition 2.1]{Gra89}. If two distinct points in $S$ belong to the same connected component of $C$, then there is a chain of rational or elliptic curves connecting them, therefore their Kobayashi distance vanishes. Conversely, by Proposition \ref{prop:GGdim2}, the Gram intersection matrix of the irreducible components of $C$ is negative definite, and every entire curve in $S$ has image contained in some connected component of $C$. 

By Grauert's criterion \cite[(2.1) Theorem 2.1, Chapter III]{BHPVdV04}), we can contract each such connected component to a point to obtain a possibly singular surface $S'$. Assume $S'$ is not Kobayashi hyperbolic and let $f\colon\mathbb{C} \to S'$ be an entire curve. Then, we can lift $f$ to an entire curve $\tilde{f}$ in $S$. 
Since the image of $\tilde{f}$ is contained in some connected component of $C$, then $f$ is constant and this gives a contradiction. 

In particular, the decreasing property of the Kobayashi pseudometric implies that if two distinct points have zero Kobayashi distance in $S$ their image must be the same in $S'$, so that they must both belong to some connected component of $C$. 
\end{proof}

\bibliography{bibliography}{}
	
\end{document}